\documentclass{amsart}

\usepackage{latexsym,amssymb,amsmath, amsfonts}
\usepackage{amsthm}
\usepackage{amscd}
\usepackage{bbm}
\usepackage{float}

\numberwithin{equation}{section}
\usepackage{enumitem,hyperref}
\hypersetup{
    colorlinks = true,
}
\newcommand{\R}{\mathbb{R}}

\newcommand{\be}{\begin{equation}}
\newcommand{\en}{\end{equation}}
\newcommand{\ee}{\end{equation}}

\DeclareMathOperator{\dist}{dist}
\DeclareMathOperator{\supp}{supp}

\DeclareMathOperator{\sign}{sign}
\newcommand{\bt}{\begin{theorem}}
\newcommand{\et}{\end{theorem}}
\newcommand{\bp}{\begin{proof}}
\newcommand{\ep}{\end{proof}}
\newcommand{\bc}{\begin{cor}}
\newcommand{\ec}{\end{cor}}
\newcommand{\bl}{\begin{lemma}}
\newcommand{\el}{\end{lemma}}
\newcommand{\bprop}{\begin{prop}}
\newcommand{\eprop}{\end{prop}}

\newtheorem{theorem}{Theorem}[section]

\newtheorem{remark}[theorem]{Remark}
\newtheorem{lemma}[theorem]{Lemma}
\newtheorem{definition}[theorem]{Definition}
\newtheorem{prop}[theorem]{Proposition}

\newtheorem{cor}[theorem]{Corollary}

\usepackage{tikz} 
\usepackage{pgfplots}
\usetikzlibrary{intersections, pgfplots.fillbetween}
\usepackage{tikz-3dplot}
\pgfplotsset{compat=1.7}
\usetikzlibrary{patterns}
\author{Ricardo. C. Freire}
\address{IMPA, Estrada Dona Castorina 110, Rio de Janeiro 22460-320, RJ Brasil}\email{rickcar8@impa.br}
\author{Argenis J. Mendez}
\thanks{A. J. Mendez was partially supported by CMM Conicyt PIA AFB170001.}
\address{Centro de Modelamiento Matemático, Universidad de Chile, Beauchef 851 Torre Norte Piso 7, Santiago de Chile}\email{amendez@cmm.uchile.cl}
\author{Oscar Ria\~no}
\address{Department of Mathematics \& Statistics, Florida International University, Miami, FL 33199, USA
}\email{ogrianoc@gmail.com}
\thanks{ Ricardo. C. Freire  was  supported by  CAPES-Brazil.}
\date{September 2020}
\title[ Benjamin-Ono-Zakharov Kuznetsov equation]{On some regularity properties for the dispersive generalized Benjamin-Ono-Zakharov-Kuznetsov Equation}

\begin{document}

\begin{abstract} 
This work aims to study some  smoothness properties  concerning   the initial value problem associated to the dispersive generalized Benjamin-Ono-Zakharov-Kuznetsov equation. More precisely, we prove  that  the solutions to this model satisfy the so-called propagation of regularity. Roughly speaking, this principle states that if the initial data enjoys some extra smoothness prescribed on a family of half-spaces, then the regularity is propagated with infinite speed. In this sense, we prove that regardless of the scale measuring the extra regularity in such hyperplane collection, then all this regularity is also propagated by solutions of this model. Our analysis is mainly based on the deduction of propagation formulas relating  homogeneous and non-homogeneous derivatives in certain regions of the plane.
\end{abstract}

\maketitle

\section{Introduction}\label{intro}
In this work we are interested in study some regularity properties for solutions of the initial value problem (IVP) associated to the \emph{dispersive generalized Benjamin-Ono-Zakharov-Kuznetsov} equation:
\begin{equation}\label{BOZK}
    \left\{\begin{aligned}
    &\partial_t u -D^{\alpha+1}_x u_x+u_{xyy}+uu_x =0,\quad (x,y,t)\in \R^{3}, \, 0\leq \alpha<1, \\
    &u(x,y,0)=u_0(x,y),
    \end{aligned}\right.
\end{equation}
where the operator $D_x^s$ denotes the  \emph{homogeneous derivative of order} $s\in \mathbb{R}$ with respect to the $x$-variable and it is  defined by 
$$D_x^s f=(c_{s}|\xi|^s \widehat{f}(\xi,\eta))^{\vee},$$
or equivalently  as $D_{x}^{s}=\left(\mathcal{H}_{x}\partial_{x}\right)^{s}$, whence $\mathcal{H}_x$ denotes the \emph{ Hilbert transform}  in the $x$-direction, that is,
\begin{equation*}   
(\mathcal{H}_xf)(x,y)=\frac{1}{\pi} \mathrm{p.v.} \int\frac{f(z,y)}{x-z} dz=\big(-i\sign(\xi)\widehat{f}(\xi,\eta)\big)^{\vee}(x,y).
\end{equation*}
%and $p.v.$ indicates the Cauchy principal value. 

The equation \eqref{BOZK} arises as a mathematical model to study the effect of dispersion on the propagation direction applied to the initial value problem for the Zakharov-Kuznetsov (ZK) equation. We recall that for $\alpha=1$, \eqref{BOZK} is the well-known \emph{Zakharov-Kuznetsov} (ZK) equation
\begin{equation}
\partial_{t}u+\partial_{x}\left(\Delta u+\frac{u^{2}}{2}\right)=0
\end{equation}
%This  equation was introduced by Zakharov and Kuznetsov in \cite{ZK} to 
that describe the propagation of ionic-acoustic waves in magnetized plasma (see \cite{ZK})
and  when $\alpha=0$, the equation \eqref{BOZK} coincides with the \emph{Benjamin-Ono-Zakharov-Kuznetsov} (BOZK)
\begin{equation}
\partial_{t}u-\mathcal{H}_{x}\partial_{x}^{2}u+\partial_{y}^{2}u+u\partial_{x}u=0,
\end{equation}
that is   presented in \cite{BOZKded,BOZKded1} as a model for thin nanoconductors on a dielectric substrate. 

The main objective of our work is to determine how dispersion affects the regularity of solutions when we  restrict the initial data  to a certain class of subsets of the Cartesian plane. This is why we consider a nonlinear model that represents a dispersive interpolation between the ZK equation and the BOZK equation. In fact, in many problems arising from Physics or Continuum Mechanics,  these models are considered  to  determine  competition between  the  nonlinearity  and the dispersion. In our case,  we are interested in  studying the propagation of regularity of solutions  of the IVP \eqref{BOZK}. Our motivation comes from the results shown   in \cite{IsazLP}, where considering suitable  solutions to the IVP associated to the $k-$generalized KdV equation, it was determined  propagation of regularity on the right-hand side (r.h.s) of the initial datum for positive times. In that sense, we  will show that the solutions of \eqref{BOZK}  satisfies this property when we restrict ourselves to an appropriated class of half-spaces determined by the dominant direction of the dispersion.

Real-valued solutions of the IVP \eqref{BOZK} (smooth enough) formally satisfy are the following conserved quantities (time invariant):
\begin{align} 
\mathcal{I}(u)&=\int_{\mathbb{R}^{2}} u(x,y,t)\, dxdy\qquad 
\mathcal{M}(u)=\int_{\mathbb{R}^{2}} u^2(x,y,t)\, dxdy, \label{mass}\\
\mathcal{E}(u)&=\frac{1}{2} \int_{\mathbb{R}^{2}} \left\{|D_x^{\frac{\alpha+1}{2}}u(x,y,t)|^2+|\partial_y u(x,y,t)|^2-\frac{1}{3}u^3(x,y,t)\right\} \, dxdy. \label{Energy}
\end{align}

To  describe  our main result, we shall  fix properly the space solution where the property described above will take place. Following \cite{Kato}, we say  that the initial value problem IVP \eqref{BOZK} is \emph{locally well-posed} (LWP) in the Banach space $X,$  if for every initial condition $u_{0}\in X$, there exists $T>0,$ and a unique solution $u(t)$ satisfying:
\begin{equation}\label{e1}
 u\in  C \left([0, T] : X\right)\cap A_{T}
 \end{equation}
where  $A_{T}$  is an auxiliary function space. Moreover, the solution map $u_{0}\longmapsto u$, is continuous from $X$ into the class \eqref{e1}. If $T$ can be taken arbitrarily large, one says that the IVP \eqref{BOZK}
 is \emph{globally well-posed}  (GWP) in the space $X.$
 
Regarding (LWP) for \eqref{BOZK} in $L^2$-type Sobolev spaces, the current best known result available was determined by Ribaud and Vento \cite{RibVento}. They addressed this question in the space $E^s$ defined by $\|f\|_{E^s}=\left\|(1+|\xi|^{\alpha+1}+\eta^2)^{s/2}\widehat{f}(\xi,\eta)\right\|_{L^2}$. It was established that \eqref{BOZK} is LWP in $E^s$ whenever $s>\frac{2}{\alpha+1}-\frac{3}{4}$ for $0\leq \alpha \leq 1$, and (GWP) in the energy space $E^{1/2}$ as soon as $\alpha >3/5$. We remark that their results are based on the short-time Bourgain spaces approach developed by Ionescu, Kenig and Tataru \cite{Ionescu2008}, combined with localized Strichartz estimates and a modified energy technique. Very recently, Cunha and Pastor in \cite{CunhaPastor} studied the Cauchy problem \eqref{BOZK} in weighted anisotropic Sobolev spaces as well as some unique continuation principles, which establish optimal spatial decay in the $x$-spatial variable.

Since our main purposes depend on techniques based on weighted energy estimates for the equation in \eqref{BOZK}, it is not clear how to address the propagation of regularity phenomena for solutions provided by the LWP result in \cite{RibVento} relaying on the short-time Bourgain spaces. Instead, we establish the following local well-posedness result which is suitable with the methods developed in this work.
\begin{theorem}\label{imprwellpos}
Assume that $0\leq \alpha < 1$ fixed. Let $s>s_{\alpha}$, where $s_{\alpha}:=(17-2\alpha)/12$. Then, for any $u_0\in H^s(\mathbb{R}^d)$, there exist a positive  time $T=T(\left\|u_0\right\|_{H^s}),$ and a unique solution $u$ to \eqref{BOZK} that belongs to
\begin{equation}\label{equclass}
    C\big([0,T];H^s(\mathbb{R}^2)\big)\cap L^1\big([0,T];W^{1,\infty}(\mathbb{R}^2)\big).
\end{equation}
Moreover, the flow map $u_0 \mapsto u(t)$ is continuous from $H^s(\mathbb{R}^2)$ to $H^s(\mathbb{R}^2)$.
\end{theorem}
%The Sobolev space $W^{1,\infty}(\mathbb{R}^2)$ is defined as usual according to the norm $\|f\|_{W^{1,\infty}}=\|f\|_{L^{\infty}}+\|\nabla f\|_{L^{\infty}}$. 
Theorem \ref{imprwellpos} is deduced by means of the short-time linear Strichartz approach method developed by Koch and Tzvetkov \cite{KochT}, and its extension given by Kenig and K\"onig \cite{KenigKo}. See \cite{OscarHBO,KenigKP,FKP} for applications to higher-dimensional models. %Since we are more concerned with special regularity properties for \eqref{BOZK}, the proof of Theorem \ref{imprwellpos} is given in the appendix.

We  shall emphasize that the Sobolev regularity attained in Theorem \ref{imprwellpos} does not yield to an improvement with respect to the conclusions in \cite{RibVento}, and to the results in \cite{CUNHA2016,ANPhD} for $\alpha=0$. Nevertheless, when $0\leq \alpha<1$, Theorem \ref{imprwellpos} states the best-known result involving solutions of \eqref{BOZK} in the class \eqref{equclass}. This conclusion is useful to deal with techniques based on energy estimates as the one we are interested in this work. %  We remark that the results in \cite{CUNHA2016,ANPhD} for the BOZK equation concerning lower regularity indexes ($s\leq 2$)  provide solutions satisfying $u_x \in L^{1}([0,T];L^{\infty}(\mathbb{R}^2))$, but not necessarily $u_y \in L^{1}([0,T];L^{\infty}(\mathbb{R}^2))$. 

Since we have  described all the requirements that our  space solution has to satisfy, we present our main result  that is summarized in the following theorem.
\begin{theorem}\label{mainTheo}
Assume $0\leq \alpha <1$ fixed. Let $u_0 \in H^{s_{\alpha}^{+}}(\mathbb{R}^2)$ where $s_{\alpha}:=(17-2\alpha)/12$. If for some $s\in \mathbb{R}$, $s>s_{\alpha}$ and $x_0 \in \mathbb{R}$
\begin{equation}\label{hypomain}
\left\|J^s_x u_0\right\|_{L^{2}_{xy}((x_0,\infty)\times \mathbb{R})}^2=\int_{-\infty}^{\infty}\int_{x_0}^{\infty}(J^s_x u_0)^2(x,y)\, dx dy <\infty,
\end{equation}
then the corresponding solution of the IVP \eqref{BOZK} provided by Theorem \ref{imprwellpos} satisfies for any $v\geq 0$ and any $\epsilon>0$,
\begin{equation}\label{e1.1}
\begin{split}
&\sup_{0\leq t \leq T}\int_{-\infty}^{\infty}\int_{x_0+\epsilon-vt}^{\infty} (J^r_x u)^2(x,y,t) \, dxdy\\ &+\int_0^T\int_{-\infty}^{\infty}\int_{x_0+\epsilon-v  t}^{x_0+\tau-v t} \big((D_x^{\frac{\alpha+1}{2}}J^s_x u)^2(x,y,t)+(\partial_y J^s_x u)^2(x,y,t) \big)\, dxdy dt\leq c,
\end{split}
\end{equation}
for all $r\in(0,s],$ where $c=c\left(\epsilon;T;v;\|u_0\|_{H^{s_{\alpha}^{+}}};\|J^s_x u_0\|_{L^2_{xy}((x_0,\infty)\times \mathbb{R}))}\right)>0.$ 
%Additionally, for any $v\geq 0$, $\epsilon>0$ and $\tau \geq 5 \epsilon$
%\begin{equation}
% \leq c,
%\end{equation}
%where $c=c\left(\epsilon;\tau;T;v;\|u_0\|_{H^{s_{\alpha}^{+}}};\|J^s_x u_0\|_{L^2_{xy}((x_0,\infty)\times \mathbb{R}))}\right)>0$.

If in addition to \eqref{hypomain},
\begin{equation}
\|J^{s+\frac{1-\alpha}{2}}u_{0}\|_{L^{2}_{xy}((x_0,\infty)\times \mathbb{R})}^2=\int_{-\infty}^{\infty}\int_{x_0}^{\infty}\left(J^{s+\frac{1-\alpha}{2}}_x u_0\right)^2(x,y)\, dx dy <\infty, 
\end{equation}
then
 for any $v\geq 0$, $\epsilon>0$ and $\tau \geq 5 \epsilon,$
\begin{equation*}
\begin{split}
&\sup_{0\leq t \leq T}\int_{-\infty}^{\infty}\int_{x_0+\epsilon-vt}^{\infty} (J^r_x u)^2(x,y,t) \, dxdy\\
& +\int_0^T\int_{-\infty}^{\infty}\int_{x_0+\epsilon-v  t}^{x_0+\tau-v t} \big((J^{s+1}_x u)^2(x,y,t)+(\partial_{y} J^{s+\frac{1-\alpha}{2}}_x u)^2(x,y,t) \big)\, dxdy dt\leq c,
\end{split}
\end{equation*}
for all $r\in\left(0,s+\frac{1-\alpha}{2}\right],$ where  the constant depends on the following parameters
 $c=c\big(\epsilon;T;v;\|u_0\|_{H^{s_{\alpha}^{+}}};\|J^{s+\frac{1-\alpha}{2}}_x u_0\|_{L^2_{xy}((x_0,\infty)\times \mathbb{R}))}\big)>0$.
\end{theorem}
The result in Theorem \ref{mainTheo} is relevant to predict the behavior of the  flow solution in terms of regularity  just by  knowing   how regular the initial data is on a particular  class of subsets of the plane. More precisely,  for $\epsilon \geq 0$ and $v\geq 0,$ we set the family of half-spaces 
\begin{equation*}
\mathfrak{H}_{\epsilon}(t):=\left\{(x,y)\in \mathbb{R}^{2}\,|\, x\geq x_{0}+\epsilon-vt\right\},\quad \, t\geq 0.
\end{equation*}
The first term on the r.h.s of \eqref{e1.1} describes the following behavior:  The regularity in the $x-$direction of $u_{0}$ in the half space $\mathfrak{H}_{\epsilon}(0),$ that is,  $J^{s}_{x}u_{0}\in L^{2}\left(\mathfrak{H}_{\epsilon}(0)\right)$  is propagated with infinite speed to the left by the flow solution. %in other words
%\begin{equation}
%J^{r}_{x}u(\cdot,t)\in L^{2}\left(\mathfrak{H}_{\epsilon}(t)\right)\subset L^{2}\left(\mathfrak{H}_0(0)\right)\quad \mbox{for all}\quad t>0,
%\end{equation}
%and all $r\in (0,s].$ 

Furthermore, the second term on the r.h.s of \eqref{e1.1} describes the extra regularity obtained in a particular class of subset of the plane, this phenomenon can be better understood by defining a new class of subsets as we did previously. More precisely, for  $\epsilon>0$ and $\tau\geq 5\epsilon,$ we define the channel
\begin{equation}
\mathcal{Q}(t):=\left\{(x,y)\in\mathbb{R}^{2}\,|\, x_{0}-\epsilon+vt <x<x_{0}+\tau-vt \right\}\quad \mbox{for all}\quad t\geq 0.
\end{equation} 
In this  setting,  the second term  on the r.h.s of \eqref{e1.1} describes the  smoothing effect of the solution  in the channel $ \mathcal{Q}(t),$ for all $t>0.$ Unlike the studied for solutions of the ZK equations (cf. \cite{ArgenisZK}), in our case, the solution enjoys of some ``anisotropic smoothing effect" it  means that  $u$ becomes smoother by one derivative in the $y-$variable when we restrict to $\mathcal{Q}(t)$ for $t>0.$ Instead, in the $x-$variable a  ``weaker"  smoothing occurs since there is only a gain of $\frac{\alpha+1}{2}$ derivatives prescribed in the channel  $\mathcal{Q}(t)$ for $t>0$. In geometrical terms, the above dynamic can be summarized in Figure \ref{fig:P1} below.
\begin{figure}
\resizebox{.65\textwidth}{!}{
\begin{tikzpicture}
 \draw[->] (0,0,0) -- (4.5,0,0) node[right] {$x$};
 \draw[dashed,-] (0,0,0) -- (-1.5,0,0);
 \draw[->] (0,0,0) -- (0,2.2,0) node[above] {};
 \draw[->] (0,0,0) -- (0,0,4.2) node[below] {$y$};
 \draw[dashed,-] (0,0,0) -- (0,0,-1);
%%Red figure
 \draw[scale=0.7,domain=1.4:7.3,smooth,variable=\x,red,name path=A] plot ({\x}, 
 {(1.3*exp(2.5*(\x-4)))/(exp(2.5*(\x-4))+exp(-2.5*(\x-4)))},5);
  \draw[scale=0.7,domain=1.4:7.3,smooth,variable=\x,red,name path=B] plot ({\x}, 
 {(1.3*exp(2.5*(\x-4)))/(exp(2.5*(\x-4))+exp(-2.5*(\x-4)))},-3);
\tikzfillbetween[of=A and B]{red, opacity=0.2};
  \node[rotate=45,text width=3cm] at (2.2,0,-2)
    {\tiny $\mathbf{x_0+\epsilon-vt}$ };
%\draw (2.33,-0.03,0) node[below] {\tiny $x_0+\epsilon+vt$};
\draw[dashdotted,green,-,name path=C] (2.4,0,3.3)--(2.4,0,-2.17);
%%%%%%%%%%% Region of regularity
\node[rotate=45,text width=3cm] at (3.8,0,-1.8)
    {\tiny $\mathbf{x_0+\tau-vt}$ };
%\draw (2.33,-0.03,0) node[below] {\tiny $x_0+\epsilon+vt$};
\draw[dashdotted,green,-,name path=D] (3.6,0,3.3)--(3.6,0,-2.17);
\tikzfillbetween[of=C and D]{pattern=horizontal lines};
%%%%%%%Travelling part
    \draw[<-,line width=0.3mm] (-0.5,0,-5.7)--(2.5,0,-5.7); 
  \draw (0.8,0.4,-6) node[below] {\small  Travelling};
%%%%% Blue figure
 \draw[scale=0.7,domain=-1:4,smooth,variable=\x,blue,name path=E] plot ({\x}, 
 {(1.3*exp(2.5*(\x-2)))/(exp(2.5*(\x-2))+exp(-2.5*(\x-2)))},5);
  \draw[scale=0.7,domain=-1:4
  ,smooth,variable=\x,blue,name path=F] plot ({\x}, 
 {(1.3*exp(2.5*(\x-2)))/(exp(2.5*(\x-2))+exp(-2.5*(\x-2)))},-3);
\tikzfillbetween[of=E and F]{blue, opacity=0.2};
\draw (0.55,-0.03,0) node[below] {\tiny $\mathbf{x_0}$};
\draw[dashdotted,blue,-] (1.02,0,3.3)--(1.02,0,-2.17);
\end{tikzpicture}}
\caption{Propagation of regularity from $\mathfrak{H}_{\epsilon}(0)$ to $\mathfrak{H}_{\epsilon}(t)$.} \label{fig:P1}
\end{figure}
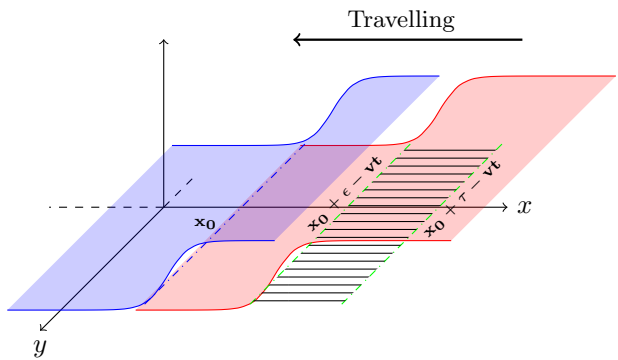

Additionally,  it is worth emphasizing several   issues  that   do not fall under the scope of Theorem \ref{mainTheo}. In comparison with our conclusions,  we notice that  for  the case of the  ZK equation ($\alpha=1$), their  solutions  propagate regularity  in both variables in a wider class of subsets of the plane (cf. \cite{ArgenisZK}). Certainly, this  contrast regarding the behavior with \eqref{BOZK} could be attributed to the differences in the nature of the fractional operator involved in the dispersion in \eqref{BOZK}, which is non-local  and tend to spread out all the information.  We think that  describing  the full behavior in both variables  requires an analysis that goes beyond  the methods employed in this paper, and it would require new tools to handle the interaction between  the operators $J^{s}$ (in the full variables) and $D_{x}^{\alpha+1}$.

The proof of Theorem \ref{mainTheo} follows in spirit the techniques and arguments presented in \cite{IsazLP,2019P1Argenis,2019P2Argenis,ArgenisZK} regarding propagation of local derivatives, and the conclusions in \cite{keLinaPOnV} for the fractional setting. However, in the case of \eqref{BOZK}, we face several additional difficulties expected from the interaction between the dispersion $\partial_xD^{\alpha+1}_x$ and the operator $J^s_x$. Among them, we require to deduce new localization formulas (see Lemma \ref{fraclemma2} below) relating the propagation on certain domains between homogeneous and non-homogeneous derivatives. This analysis is provided by studying the kernel determined by the difference $J^s_{x}-D^s_{x}$ as well as examining some class of pseudo-differential operators. In fact, we believe that these localization formulas are of independent interest and could certainly be applied to a wider class of equations in arbitrary spatial dimension. %For an application involving differential operators and the aforementioned kernel, we refer to the work of Bourgain and Li \cite{bourgain2014}. 
\begin{remark}
In the case of physical relevance  $\alpha=0$ in \eqref{BOZK}, i.e., the BOZK equation,  Theorem \ref{mainTheo} leads to an extension to the fractional setting of the conclusions derived in \cite[Theorem 1.4]{ANPhD} concerning the propagation of regularity principle for local derivatives. 	
\end{remark}
\begin{remark}
The method of  proof of Theorem \ref{mainTheo} applies for a given initial data $u_0\in H^{r}(\mathbb{R}^2)$ with arbitrary regularity $r>0$ provided that one can assure the existence of a corresponding solution $u\in C([0,T],H^r(\mathbb{R}^2))$ of \eqref{BOZK} such that
\begin{equation*}
u, \partial_x u \in L^1([0,T];L^{\infty}(\mathbb{R}^2)).
\end{equation*}
In particular, in the case of the BOZK equation, $\alpha=0$ in \eqref{BOZK}, the LWP result in \cite[Theorem 1.3]{ANPhD} determines the validity of Theorem \ref{mainTheo} for initial data with regularity $H^r(\mathbb{R}^2)$, $r>5/4$.
\end{remark}
\begin{remark}
We believe that our results can be adapted to study the  propagation of regularity  principle for other two-dimensional models involving non-local operators. For instance, we expect to obtain similar conclusion to that of Theorem \ref{mainTheo} for solutions of the Cauchy problem associated to the fractional Kadomtsev Petviashvili-equations (KP) type equation
    \begin{equation*}
    \begin{aligned}
    &\partial_t u -\partial_xD^{\alpha}_x u+\kappa \partial_{x}^{-1}\partial_y^2u+uu_x =0,\quad (x,y,t)\in \mathbb{R}^{3},
    \end{aligned}
\end{equation*}
where $0<\alpha\leq 2$, see \cite{FKP}. %This equation can be thought as a two-dimensional weakly transverse version of the fractional KdV equation. %We refer to \cite{FKP} for a more detailed discussion on its derivation and to consult some well and ill-posedness conclusions for the IVP associated. 
Additionally, we expect that our considerations may work as an initial step to obtain fractional propagation of regularity for the IVP associated to the Shira equation  
\begin{equation*}
  \partial_t u -\mathcal{H}_x\partial_x^2u-\mathcal{H}_x\partial_y^2u+u\partial_{x} u=0,\hskip 15pt (x,y,t)\in \mathbb{R}^3.
\end{equation*}
This equation was deduced as a simplified model to describe a two-dimensional weakly nonlinear long-wave perturbation on the background of a boundary-layer type plane-parallel shear flow (see \cite{PShrira}). For some references dealing with LWP issues see \cite{Paisas1,OscarCapila}.
\end{remark} 
The document is organized as follows: we first introduce the notation to be employed through our analysis. Section \ref{preliminar} is aimed to present some commutator estimates, pseudo-differential operators as well as  results involving localized regularity. In this section, we  also introduce the weighted functions required to apply energy estimates. Next, in Section \ref{Katosect}, we turn our attention  to the deduction of Kato's smoothing effect, which provides a set up in the study of the propagation of regularity. Section \ref{mainTheosection} is intended to deduce our main result: Theorem \ref{mainTheo}. Finally,  we close with an appendix aimed to provide the local well-posedness result stated in Theorem \ref{imprwellpos}.

%%%%%%%%%%%%%%%%%%%%%%%%%%%%%%%%%%%%%%%%%%%%%%%%%%%%%%%%%%%%%%%%%%%%%%%%%%%%%%%%%%%%%%%%%%%%%%%%%%%%%%%%%%%%%%%%%%%%%%%%%%%%%%%%%%%%%%%%%%%%%%%%%%%%%%%%%%%%%%%%%%%%%%%%%%%%%%%%%%%%%%%%%%%%%%%%%%%%%%%%%%%%%%%%%%%%%%%%%%%%%%%%%%%%%%
\section{Notation}
Given two positive quantities $a$ and $b$,  $a\lesssim b$ means that there exists $c>0$ such that $a\leq c b$. $[A,B]$ denotes the \emph{commutator} between the operators $A$ and $B$, that is
$$[A,B]=AB-BA.$$
We shall employ the standard multi-index notation, $\gamma=(\gamma_1,\dots,\gamma_d) \in \mathbb{N}^d$, $\partial^{\gamma}=\partial^{\gamma_1}_{x_1}\cdots \partial_{x_d}^{\gamma_d}$, $|\gamma|=\sum_{j=1}^d \gamma_j$, and $\gamma!=\gamma_1 ! \cdots \gamma_d !$.

Given $p\in [1,\infty]$ and $d\geq 1$ integer, $L^p(\mathbb{K})$ or simply $L^p$ denotes the usual \emph{Lebesgue space}, where the set $\mathbb{K}$ will be easily deduced in each context. To emphasize the dependence on the variables when $d=2$, we will denote by $\|f\|_{L^p(\mathbb{R}^2)}=\|f\|_{L^p_{xy}(\mathbb{R}^2)}$.  We denote by $C_c^{\infty}(\mathbb{R}^d)$ the spaces of \emph{smooth functions with compact support} and $S(\mathbb{R}^d)$ the space of \emph{Schwarz functions}.

The operators $D^s=(-\Delta)^{s/2}$ and $J^s=(I-\Delta)^{s/2}$ denote the \emph{Riesz and Bessel potentials} of order $-s$, respectively. As defined above, $D_x^sf$ and $D_y^sf$ denote the operators
\begin{equation*}
    \widehat{D_x^s f}(\xi,\eta)=|\xi|^s\widehat{f}(\xi,\eta) \text{ and } \widehat{D_y^s f}(\xi,\eta)=|\eta|^s\widehat{f}(\xi,\eta).
\end{equation*}
Analogously, $J_x^sf$ and $J_y^sf$ are determined by
\begin{equation*}
    \widehat{J_x^s f}(\xi,\eta)=(1+|\xi|^2)^{s/2}\widehat{f}(\xi,\eta) \text{ and } \widehat{J_y^s f}(\xi,\eta)=(1+|\eta|^2)^{s/2}\widehat{f}(\xi,\eta).
\end{equation*}
Given $s\in \mathbb{R}$, $H^s(\mathbb{R}^d)$ denotes the $L^2$-based Sobolev space with norm $\|f\|_{H^s} =\|J^s f\|_{L^2}$. If $B$ is a space of functions on $\mathbb{R}$, $T>0$ and $1\leq p \leq \infty$, we define the spaces $L^p_TB$ and $L^p_tB$ by the norms 
\begin{equation*}
 \|f\|_{L^p_T B} =\|\|f(\cdot,t)\|_{B}\|_{L^p([0,T])} \, \text{ and } \,    \|f\|_{L^p_tB} = \|\|f(\cdot,t)\|_{B}\|_{L^p(\mathbb{R})}.
\end{equation*}

%%%%%%%%%%%%%%%%%%%%%%%%%%%%%%%%%%%%%%%%%%%%%%%%%%%%%%%%%%%%%%%%%%%%%%%%%%%%%%%%%%%%%%%%%%%%%%%%%%%%%%%%%%%%%%%%%%%%%%%%%%%%%%%%%%%%%%%%%%%%%%%%%%%%%%%%%%%%%%%%%%%%%%%%%%%%%%%%%%%%%%%%%%%%%%%%%%%%%%%%%%%%%%%%%%%%%%%%%%%%%%%%%%%%%%%%%%%%

\section{Preliminaries}\label{preliminar}

This section is aimed to present the preliminaries and initial considerations required to develop our arguments.

\subsection{Commutator estimates}

To obtain estimates for the nonlinear terms, the following results will be implemented along our considerations.
\begin{lemma}\label{conmKP}
If $s>0$ and $1<p<\infty$, then
\begin{equation}\label{conmKPI}
    \left\|[J^s,f]g\right\|_{L^p(\mathbb{R}^d)} \lesssim  \left\|\nabla f\right\|_{L^{\infty}(\mathbb{R}^d)} \left\|J^{s-1}g\right\|_{L^p(\mathbb{R}^d)}+ \left\|J^s f\right\|_{L^p(\mathbb{R}^d)} \left\|g\right\|_{L^{\infty}(\mathbb{R}^d)}.
\end{equation}
\end{lemma}
Lemma \ref{conmKP} was proved initially  by Kato and Ponce in \cite{KP}. See also \cite{Smoothcomm} and the references therein.

Additionally, we recall the following commutator relation for non-homogeneous derivatives
 \begin{lemma} \label{nonhest}
Let $s>0$, $1<p<\infty$ and $1<p_1,p_2,p_3,p_4 \leq \infty$ satisfying
\begin{equation*}
    \frac{1}{p}=\frac{1}{p_1}+\frac{1}{p_2}=\frac{1}{p_3}+\frac{1}{p_4}.
\end{equation*}
Then,
\begin{itemize}
    \item[(i)] if $0<s\leq 1$,
    \begin{equation*}
        \left\|[D^s,f]g\right\|_{L^p(\mathbb{R}^d)} \lesssim \left\|D^{s-1}\nabla f\right\|_{L^{p_1}(\mathbb{R}^d)}\left\|g\right\|_{L^{p_2}(\mathbb{R}^d)}.
    \end{equation*}
    \item[(ii)] If $s>1$, then 
    \begin{equation*}
        \left\|[D^s,f]g\right\|_{L^p(\mathbb{R}^d)} \lesssim \left\|D^{s-1}\nabla f\right\|_{L^{p_1}(\mathbb{R}^d)}\left\|g\right\|_{L^{p_2}(\mathbb{R}^d)}+\left\|\nabla f\right\|_{L^{p_3}(\mathbb{R}^d)}\left\|D^{s-1}g\right\|_{L^{p_4}(\mathbb{R}^d)}.
    \end{equation*}
\end{itemize}
\end{lemma}
The above estimates were deduced by D. Li in \cite[Corollary 5.3]{Dli}.
\begin{lemma}\label{fraLR}
Given  $s>0.$ Let     $1<p_1,p_2,q_1,q_2  \leq \infty$  
with $$\frac{1}{p_j}+\frac{1}{q_j}=\frac{1}{2}.$$
Then,
\begin{equation}
\begin{split}
  \left\|D^s(fg)\right\|_{L^2(\mathbb{R}^d)} &\lesssim  \left\|D^s f\right\|_{L^{p_1}(\mathbb{R}^d)} \left\|g\right\|_{L^{q_1}(\mathbb{R}^d)}+ \left\|f\right\|_{L^{p_2}(\mathbb{R}^d)} \left\|D^sg\right\|_{L^{q_2}(\mathbb{R}^d)} \label{libRule} \\
\end{split}
    \end{equation}
    and 
    \begin{equation}
    \begin{split}
     \left\|J^s(fg)\right\|_{L^2(\mathbb{R}^d)} & \lesssim  \left\|J^s f\right\|_{L^{p_1}(\mathbb{R}^d)} \left\|g\right\|_{L^{q_1}(\mathbb{R}^d)}+ \left\|f\right\|_{L^{p_2}(\mathbb{R}^d)} \left\|J^s g\right\|_{L^{q_2}(\mathbb{R}^d)}. \label{eqfraLR}
    \end{split}
    \end{equation}    
\end{lemma}
The proof of the above estimates  can be   consulted in \cite{FRAC}.
 \begin{lemma}\label{lem1}
  Let $\phi\in  C^{\infty}(\mathbb{R})$  with  $\phi'\in C_{0}^{\infty}(\mathbb{R})$. If $s\in \mathbb{R}$, then for any
 	$l>|s-1|+1/2$
\begin{equation}
\|[J^{s}, \phi] f\|_{L^{2}(\mathbb{R})}+\|[J^{s-1},\phi]\partial_{x}f\|_{L^{2}(\mathbb{R})}\leq c\|J^{l}\phi'\|_{L^{2}(\mathbb{R})}\|J^{s-1}f\|_{L^{2}(\mathbb{R})}.
\end{equation}
 \end{lemma}
The previous lemma was established by Kenig, Linares, Ponce and Vega in  \cite{keLinaPOnV}. 

We shall employ the following generalization of Calderon's first commutator estimate in the context of the Hilbert transform deduced in \cite[Lemma 3.1]{DawsonMCPON} (see also \cite[Proposition 3.8]{Dli}).
\begin{prop}\label{CalderonComGU}
Let $1<p<\infty$ and $l,m \in \mathbb{Z}^{+}\cup \{0\}$, $l+m\geq 1,$ then
\begin{equation}\label{Comwellprel1}
    \|\partial_x^l[\mathcal{H}_x,g]\partial_x^{m}f\|_{L^p(\mathbb{R})} \lesssim_{p,l,m} \|\partial_x^{l+m} g\|_{L^{\infty}(\mathbb{R})}\|f\|_{L^p(\mathbb{R})}.
\end{equation}
\end{prop}

%%%%%%%%%%%%%%%%%%%%%%%%%%%%%%%%%%%%%%%%%%%%%%%%%%%%%%%%%%%%%%%%%%%%%%%%%%%%%%%%%%%%%%%%%%%%%%%%%%%%%%%%%%%%%%%%%%%%%%%%%%%%%%%%%%%%%%%%%%%%%%%%%%%%%%%%%%%%%%%%%%%%%%%%%%%%%%%%%%%%%%%%%%%%%%%%%%%%%%%%%%%%%%%%%%%%%%%%%%%%%%%%%%%%%%%%%%%%%%%%%%%%%%%%%

\subsubsection*{\bf Commutator expansion}

Now, we present several commutator expansions for the operator $[-\mathcal{H}_xD_x^a,h]$ in one-dimensional variable.  These results are due to Ginibre and Velo in \cite{GiniVelo,GiniVelo2}.

We first introduce some additional notation. Let $a=2\mu+1>1$, $n$ be a non-negative integer and $h$ be a smooth function with suitable decay at infinity, for instance,  $h'\in C^{\infty}_{0}(\mathbb{R})$. We define the operators
\begin{equation}\label{commutremind}
    R_n(a):=[HD^a_x,h]-\frac{1}{2}(P_n(a)-HP_n(a)H),
\end{equation}
where $H=-\mathcal{H}_x$ and
\begin{equation}\label{GVexpression}
    P_n(a):=a\sum_{0\leq j\leq n} c_{2j+1} (-1)^j D_x^{\mu-j}h^{2j+1}D_{x}^{\mu-j}
\end{equation}
with 
\begin{equation}
    c_1=1 \text{ and } c_{2j+1}=\frac{1}{(2j+1)!} \prod_{0\leq k\leq j}(a^2-(2k+1)^2).
\end{equation}
The next proposition establishes continuity properties for the operator $R_n(a)$.
\begin{prop}\label{valueofn}
Let $n$ be a non-negative integer, $a\geq 1$, and $b\geq 0$ be such that
\begin{equation}
    2n+1\leq a+2b\leq 2n+3.
\end{equation}
Then,
\begin{itemize}
    \item[(i)] the operator $D_x^b R_n(a) D^a_x$ is bounded in $L^2(\mathbb{R})$ and satisfy
    \begin{equation}\label{Restestimate}
        \|D^b_x R_n(a)D^b_x f\|_{L^2} \lesssim \|\widehat{D_x^{a+2b}h}\|_{L^1_{\xi}} \|f\|_{L^2}.
    \end{equation}
 %   In particular if $a\geq 2n+1$, one can take the implicit constant $c=1$.
    \item[(ii)] If in addition $a+2b< 2n+3$, then the operator $D_x^b R_n(a) D^a_x$ is compact in $L^2(\mathbb{R})$.
\end{itemize}
\end{prop}
\begin{proof}
See \cite[Proposition 2.2 ]{GiniVelo2}.
\end{proof}

%%%%%%%%%%%%%%%%%%%%%%%% Colocar Referência Bourgan-Dong li. %%%%%%%%%%%%%%%%%%%%%%%  
\subsection{Pseudo-differential Operators}
To facilitate the exposition of our results, this subsection is intended to briefly indicate some preliminaries results  concerning pseudo-differential operators, as well as, some consequences of them. %and  some  results that will be important in the next section  % required.% for our considerations.
\begin{definition}
Let $a \in C^{\infty}(\mathbb{R}^d \times \mathbb{R}^d)$ satisfying %the differential inequalities 
\begin{equation*}
|\partial^{\beta}_x\partial^{\gamma}_{\xi}a(x,\xi)| \leq A_{\gamma,\beta}(1+ |\xi|)^{m-|\gamma|},    
\end{equation*}
for some $m \in \mathbb{R}^d$ and for all the multi-index $\gamma$ and $\beta$. This function a will be called a symbol of order m and $\mathcal{S}^m(\mathbb{R}^d \times \mathbb{R}^d)$, simplifying $\mathcal{S}^m$ will represent the set of these type of functions.
\end{definition}

\begin{definition}
A pseudo-differential operator is a mapping $f \mapsto \Psi_a f$ given by
\begin{equation*}
(\Psi_af)(x) = \int_{\mathbb{R}^n} e^{2\pi i x\cdot\xi} a(x,\xi)\widehat{f}(\xi)\, d\xi,    
\end{equation*}
where $a(x, \xi)$ is the symbol of $\Psi_a$. 
\end{definition}

After these definitions, we will need the following theorem.
\begin{theorem}\label{Cont.PSDO}
Let $a$ be a symbol of order $0$, i.e., $a \in \mathcal{S}^0$. Then, the operator $\Psi_a$, initially defined on $S(\mathbb{R}^d)$ can be extended to a bounded operator from $L^2(\mathbb{R}^d)$ to itself.   
\end{theorem}

\begin{proof}
The proof can be consulted in \cite[Chapter VI, Theorem 1]{stein1993harmonic}. 
\end{proof}

For a generalization of Theorem \ref{Cont.PSDO}, we refer to \cite[Chapter 3, Theorem 3.6]{raymond1991elementary}. A key ingredient in our considerations is the following kernel representation for pseudo-differential operators.
\begin{prop}\label{Kernelpseudo}
Let $a \in \mathcal{S}^m ,$  and $\Psi_a$ its associated pseudo-differential operator. Then, there exists a kernel $k_a\in C^{\infty}(\mathbb{R}\times (\mathbb{R}^d \setminus\{0\}))$ satisfying the following properties:
\begin{itemize}
\item[(i)] The operator  admits the following representation
$$\Psi_a f(x)=\int k_a(x,x-y)f(y) \, dy,$$
for all $x\notin \supp(f)$,
\item[(ii)] additionally, for all multi-indices $\gamma$ and $\beta$, and all $N \geq 0,$ integer, it follows
  $$|\partial_x^{\beta}\partial_z^{\gamma}k_a(x,z)| \leq A_{\gamma, \beta, N}|z|^{-d-m-|\gamma|-N}, \quad z \neq 0,$$    
whenever $d+m+|\gamma| + N > 0.$ 
\end{itemize}
\end{prop}
The following result will be useful to approximate the composition between pseudo-differential operators.

\begin{prop}\label{PO1}
Let $a$ and $b$ symbols belonging to $\mathcal{S}^{r_{1}}$ and $\mathcal{S}^{r_{2}}$ respectively. Then, there is a symbol $c \in \mathcal{S}^{r_{1} + r_{2}}$ so that 
\begin{equation*}
    \Psi_c = \Psi_a\circ\Psi_b.
\end{equation*}
Moreover, 
\begin{equation}
    c \sim \sum_{\beta}\frac{1}{(2\pi i)^{|\beta|} \beta !}(\partial^{\beta}_{\xi}a)\cdot(\partial^{\beta}_{x}b),
\end{equation}
in the sense that 
\begin{equation*}
    c - \sum_{|\beta|< N}\frac{(2\pi i)^{-|\beta|}}{\beta !}(\partial^{\beta}_{\xi}a)\cdot(\partial^{\beta}_{x}b) \in \mathcal{S}^{r_{1} + r_{2} - N}
\end{equation*}
for all $N \geq 0$.
\end{prop}
Another important consequence regarding pseudo-differential operators is the following symbolic calculus for commutators.
\begin{prop}\label{PO2}
For $a \in \mathcal{S}^{r_{1}}$ and $b \in \mathcal{S}^{r_{2}} $ we define the commutator $[\Psi_a,\Psi_b]$ by
\begin{equation*}
    [\Psi_a,\Psi_b] = \Psi_a \circ \Psi_b - \Psi_b \circ \Psi_a .
\end{equation*}
Then, the symbol of the operator %$[\Psi_a,\Psi_b] \in \mathcal{S}^{r + s - 1}$ 
is given by 
\begin{equation}
    c = \frac{1}{i} \sum_{j=1}^{d} \left( \frac{\partial a}{\partial_{\xi_j}}\frac{\partial b}{\partial_{x_j}} - \frac{\partial a}{\partial_{x_j}}\frac{\partial b}{\partial_{\xi_j}} \right) \mbox{ mod } \mathcal{S}^{r_{1} + r_{2} - 2}.
\end{equation}
\end{prop}
%Next, we present some consequences of the previous results. We begin by showing a proposition relating $J^s(f\theta)$ with $(J^sf)\theta$. 
%Now, we  state a proposition
Now, we present  through  the following result  some relations  between $J^{s}(\theta f)$ and $\theta J^{s}f.$
\begin{prop}\label{PO3}
Let $s>0$ and $m,d$ be positive integers such that $m\geq \max\{s,d\}$. Additionally, we  consider $\theta \in C^{\infty}(\mathbb{R}^d)$, $0\leq \theta \leq 1,$ satisfying that $\partial^{\gamma}\theta \in L^{\infty}(\mathbb{R}^d)$ for any multi-index $\gamma$. Then, there exist some constants $c_{\gamma}$, pseudo-differential operators of order zero $\Psi^{\gamma}$ for each multi-index $0\leq |\gamma|\leq m$, and $K_{s-m}$ of order $s-m$ such that
\begin{equation}
J^{s}(\theta f)=\sum_{1\leq |\gamma|\leq m} c_{\gamma}\partial_{x}^{\gamma} \theta \, \Psi^{\gamma}(J^{s-|\gamma|}f)+\theta J^{s}f +K_{s-m}f,
\end{equation}
provided that $f$ is  regular enough.  
\end{prop}
\begin{proof}
In virtue  of the  identity
\begin{equation*}
J^{s}(\theta f)=[J^s,\theta]f+ \theta J^{s}f,
\end{equation*}
we are reduced to decompose the operator $[J^s,\theta]$. Thus, by employing Propositions \ref{PO1} and \ref{PO2}, we find
\begin{equation*}
[J^{s},\phi]f(x)=\int a_{s}(x,\xi) \widehat{f}(\xi)e^{i x\cdot\xi}\, d\xi,
\end{equation*}
where 
\begin{equation}\label{eqmainthe6}
a_{s}(x,\xi) =\sum_{1\leq |\gamma|\leq m} c_{\gamma} \partial_{\xi}^{\gamma}\langle \xi\rangle^{s}\partial_{x}^{\gamma}\theta(x)+k_{s-m},
\end{equation}
and $k_{s-m} \in \mathcal{S}^{s-m} \subset \mathcal{S}^{0}$. Now, since
\begin{equation*}
\begin{aligned}
\sum_{1\leq |\gamma|\leq m}& c_{\gamma}\partial_{\xi}^{\gamma} (\langle \xi\rangle^{s})\partial_{x}^{\gamma}\theta(x)=\sum_{1\leq |\gamma| \leq m} c_{\gamma}\frac{\partial_{\xi}^{\gamma}(\langle \xi\rangle^{s})}{\langle \xi\rangle^{s-|\gamma|}}(\langle \xi\rangle^{s-|\gamma|}\partial_{x}^{\gamma}\theta(x)),
\end{aligned}
\end{equation*}
we are led to define the pseudo-differential operator $\Psi^{\gamma}$ according to the symbol 
\begin{equation*}
\zeta^{\gamma}(\xi):=\langle \xi\rangle^{|\gamma|-s}\partial_{\xi}^{\gamma}(\langle \xi\rangle^{s}) \in \mathcal{S}^0,
\end{equation*}
for each $0\leq |\gamma|\leq m$. Gathering the preceding results, it follows
\begin{equation*}
\begin{aligned}
\left[J^{s},\theta\right]f=\sum_{1\leq |\gamma|\leq m} c_{\gamma}\partial_{x}^{\gamma} \theta \Psi^{\gamma}(J^{s-\gamma}f)+K_{s-m}f,
\end{aligned}
\end{equation*} 
where $K_{s-m}$ is the pseudo-differential operator with symbol $k_{s-m} \in \mathcal{S}^{s-m}$ defined as in \eqref{eqmainthe6}. This completes the proof.
\end{proof}

The kernel representation of pseudo-differential operators has been applied to obtain some regularity properties for the product of functions with separated supports. In this regard, the following result was deduced in \cite{ArgenisZK}.
\begin{lemma}\label{supportseparetedJ}
Let $\gamma$ be a multi-index and $\Psi_{a}$ a pseudo-differential operator  of order  $m.$ If $g\in L^2(\mathbb{R}^d)$ and $f\in L^p(\mathbb{R}^d)$, $p \in [2,\infty]$ with
    \begin{equation*}
    \dist(\supp(f), \supp(g)) \geq \delta > 0,
    \end{equation*}
then,
    \begin{equation*}
    \|f\partial^{\gamma}_x\Psi_{a} g\|_{L^2} \lesssim \|f\|_{L^p}\|g\|_{L^2}.
    \end{equation*}
\end{lemma}
We also require some fractional version of the above lemma. For that reason, we are interested in investigating some interactions between the non-local operators $D^{s}$ and $J^s$. Broadly speaking, by taking advantage of the kernel obtained by the difference $J^s-D^s$, the idea is to transfer localization properties between homogeneous and non-homogeneous derivatives. We refer to \cite{bourgain2014} for an application dealing with the difference $J^s-D^s$. 
\begin{lemma}\label{supportsepareted2}
Let $s\in \mathbb{R}$, $s_1\in (0,1)$.  If $f \in L^{\infty}(\mathbb{R}^d) $ and $g \in L^{2}(\mathbb{R}^d)$ with 
    \begin{equation*}
    \dist(\supp(f),\supp(g)) \geq \delta > 0 
    \end{equation*}
Then, 
\begin{equation*}
\|f D^{s_1}J^{s}g\|_{L^2}\lesssim \|f\|_{L^{\infty}}\|g\|_{L^2}.
\end{equation*}
\end{lemma}
\begin{proof}
We begin by choosing an integer $M>1,$ such that $M>(s+s_1)/2$. By employing the binomial expansion, we get
\begin{equation*}
\begin{aligned}
\langle \xi \rangle^{s}\big(\langle \xi \rangle^{s_1}-|\xi|^{s_1}\big)&=\langle \xi \rangle^{s+s_1}(1-(1-\langle  \xi\rangle^{-2})^{s_1/2}) \\
&=\sum_{j=1}^{M-1} \binom{s_1/2}{j}\frac{(-1)^{j+1}\langle \xi \rangle^{s+s_1}}{\langle \xi \rangle^{2j}}+\sum_{j=M}^{\infty} \binom{s_1/2}{j}\frac{(-1)^{j+1}}{\langle \xi \rangle^{2j-s-s_1}} \\
&=:k_{M}(\xi)+\sum_{j=M}^{\infty} \binom{s_1/2}{j}\frac{(-1)^{j+1}}{\langle \xi \rangle^{2j-s-s_1}},
\end{aligned}
\end{equation*}
where given that $s_1>0$, the above series converges absolutely. Thus, we set the operator $K_{M}$ by
\begin{equation*}
K_{M}f(x)=\int k_{M}(\xi)\widehat{f}(\xi) e^{2\pi i x\cdot \xi}\, d\xi.
\end{equation*}
Consequently, we deduce the identity
\begin{equation}\label{eqlemm1}
\begin{aligned}
f D^{s_1}J^{s}g=fJ^{s+s_1}g-fK_{M}g-\sum_{j=M}^{\infty}\binom{s_1/2}{j}(-1)^{j+1}(fJ^{2j-s-s_1}g).
\end{aligned}
\end{equation}
We are led to estimate each factor on the right-hand side of the above equality. Since the Bessel potential satisfies $\|J^{-s'}g\|_{L^p}\leq \|g\|_{L^{p}},$ for any $1\leq p \leq \infty,$ and $s'>0$, we deduce from our choice of $M$,
\begin{equation*}
\begin{aligned}
\|\sum_{j=M}^{\infty}\binom{s_1/2}{j}(-1)^{j+1}(fJ^{2j-s-s_1}g)\|_{L^2} &\leq \sum_{j=M}^{\infty}\Big|\binom{s_1/2}{j}\Big|\|f\|_{L^{\infty}}\|J^{2j-s-s_1}g\|_{L^2} \\
&\lesssim \|f\|_{L^{\infty}}\|g\|_{L^2}.
\end{aligned}
\end{equation*}
Next, by the hypothesis between the supports of $f$ and $g$, the estimate for $fJ^{s+s_1}g$ is a consequence of Lemma \ref{supportseparetedJ}. Likewise, noticing that $K_{M}$ defines a pseudo-differential operator of order $s+s_1-2$, the required estimate is again a consequence of Lemma \ref{supportseparetedJ}. Summarizing,
\begin{equation*}
\begin{aligned}
\|f J^{s+s_1}g\|_{L^2}+\|fK_{M}g\|_{L^2}\lesssim \|f\|_{L^{\infty}}\|g\|_{L^2}.
\end{aligned} 
\end{equation*}
Going back to identity \eqref{eqlemm1}, we gather the previous results to complete the proof of the lemma. 
\end{proof}

The same arguments in the proof of Lemma \ref{supportsepareted2} provide the following generalization.

\begin{cor}\label{supportsepareted3}
Let $\Psi_a$ be a pseudo-differential operator and $s_1\in [0,1)$.  If $f \in L^{\infty}(\mathbb{R}^d) $ and $g \in L^{2}(\mathbb{R}^d)$ are such that
    \begin{equation*}
    \dist(\supp(f),\supp(g)) \geq \delta > 0 
    \end{equation*}
Then, 
\begin{equation*}
\|f D^{s_1}\Psi_a g\|_{L^2}\lesssim \|f\|_{L^{\infty}}\|g\|_{L^2}.
\end{equation*}
\end{cor}

%%%%%%%%%%%%%%%%%%%%%%%%%%%%%%%%%%%%%%%%% ADD LOCAL OF THE PROOFS%%%%%%%%%%%%%%%%%%%%%%%%%%%%%%%%%%%%%%%%%%%%%%%%%%%

\subsection{Localized Regularity}

This subsection introduces the main tool required to deduce Theorem \ref{mainTheo}. Mainly, the idea is to provide formulas connecting the propagation of regularity effect in different domains. Estimates of this kind were previously presented in \cite{keLinaPOnV} and \cite{ArgenisZK}. A contribution of the present work is the deduction of (I) and (II) below, which connect the operators $D^s$ and $J^s$ in different regions. % We will observe that these estimates are deduced by bounding the kernel determined by the difference $D^s-J^s$.
 Additionally, we emphasize that similar estimates (III) and (IV) were previously determined in the work of \cite{ArgenisZK}, but here, we perform some minor changes to consider functions $f\in H^{-m}(\mathbb{R}^d)$, $m\geq 0$.

\begin{lemma}\label{fraclemma2}
Let $f \in H^{-m}(\mathbb{R}^d)$ for some integer $m\geq 0$, and $\theta_1, \theta_2\in C^{\infty}(\mathbb{R}^d)\setminus\{0\}$ such that $0\leq \theta_1,\theta_2 \leq 1$,
\begin{equation}\label{assumseparsupp}
\dist(\supp(1-\theta_1),\supp(\theta_2))\geq \delta>0,
\end{equation}
and satisfying $\partial^{\gamma}\theta_1, \partial^{\gamma}\theta_2 \in L^{\infty}(\mathbb{R}^d)$ for all multi-index $\gamma$.
\begin{itemize}
\item[(I)] If $0\leq \beta <2$ and $\theta_1 f , \theta_1 D^{\beta}f \in L^{2}(\mathbb{R}^d)$, then
\begin{equation*}
\|\theta_2 J^{\beta}f\|_{L^2} \lesssim \|\theta_1f \|_{L^2}+\|\theta_1 D^{\beta}f \|_{L^2}+\|J^{-m}f\|_{L^2},
\end{equation*}
so that $\theta_2 J^{\beta}f  \in L^{2}(\mathbb{R}^d)$.
\item[(II)] If $0\leq \beta <2$, and $\theta_1 J^{\beta}f  \in L^{2}(\mathbb{R}^d)$, then  
\begin{equation*}
\|\theta_2 f\|_{L^2}+\|\theta_2 D^{\beta}f\|_{L^2}\lesssim \|\theta_1 J^{\beta}f\|_{L^2}+\|J^{-m}f\|_{L^2},
\end{equation*}
and so $\theta_2 f, \theta_2 D^{\beta}f \in L^{2}(\mathbb{R}^d)$.
\item[(III)] If $s>0$, $0\leq r \leq s$, and $\theta_1 J^{s}f  \in L^{2}(\mathbb{R}^d)$, then  
\begin{equation*}
\|\theta_2 J^{r}f\|_{L^2}\lesssim \|\theta_1 J^{s}f\|_{L^2}+\|J^{-m}f\|_{L^2},
\end{equation*}
and so $\theta_2 J^{r}f \in L^{2}(\mathbb{R}^d)$.
\item[(IV)] If $s > 0$ and $\theta_1J^sf \in L^2(\mathbb{R}^d)$, then
\begin{equation*}
\| J^s(\theta_2 f)\|_{L^{2}} \lesssim \|\theta_1J^sf\|_{L^2}+\|J^{-m}f\|_{L^2},
\end{equation*}
that is, $J^s(\theta_2 f)\in L^2(\mathbb{R}^d)$.
\end{itemize}
\end{lemma}

\begin{proof}
We first deduce (I). We begin by analyzing the difference between $J^{\beta}-D^{\beta}$ as it was done in the proof of Lemma \ref{supportsepareted2}. For this purpose, let us consider some integer $M>1$ fixed such that $M> 2m+\beta/2$, where $m$ is such that $f\in H^{-m}(\mathbb{R}^d)$. By means of the binomial expansion, we have
\begin{equation}\label{eqlem1}
\begin{aligned}
\langle \xi \rangle^{\beta}-|\xi|^{\beta}&=\langle \xi \rangle^{\beta}(1-(1-\langle  \xi\rangle^{-2})^{\beta/2}) \\
&=\sum_{j=1}^{M-1} \binom{\beta/2}{j}\frac{(-1)^{j+1}}{\langle \xi \rangle^{2j-\beta}}+\sum_{j=M}^{\infty} \binom{\beta/2}{j}\frac{(-1)^{j+1}}{\langle \xi \rangle^{2j-\beta}} \\
&=:k_{1,M}(\xi)+k_{2,M}(\xi),
\end{aligned}
\end{equation}
where given that $\beta>0$, the above series converges absolutely. Thus, we define the operators $K_{j,M}$, $j=1,2$ by
\begin{equation*}
K_{j,M}f(x)=\int k_{j,M}(\xi)\widehat{f}(\xi) e^{i x\cdot \xi}\, d\xi, \hspace{0.4cm} j=1,2.
\end{equation*}
In virtue of  \eqref{eqlem1}, we write
\begin{equation*}
\theta_2 J^{\beta}f =\theta_2 D^{\beta} f +\theta_2 K_{1,M}f+\theta_2K_{2,M}f.
\end{equation*}
Consequently, to deduce (I), we are reduced to prove $\theta_2 K_{j,M}f \in L^2(\mathbb{R}^d)$, $j=1,2$. To deal with the estimate concerning the first operator $K_{1,M}$, we perform the following decomposition
\begin{equation*}
\begin{aligned}
\theta_2 K_{1,M}f =  \theta_2 K_{1,M}(\theta_1 f )+\theta_2 K_{1,M}((1- \theta_1)f).
\end{aligned}
\end{equation*}
Since $0\leq \beta <2$, we have that $k_{1,M} \in \mathcal{S}^{0}$, i.e., $K_{1,M}$ determines a pseudo-differential operator of order zero. Thus, Theorem \ref{Cont.PSDO} yields
\begin{equation}\label{eqlem2}
\|\theta_2 K_{1,M}(\theta_1 f ) \|_{L^2} \lesssim \|\theta_2\|_{L^{\infty}}\|\theta_1 f \|_{L^2}.
\end{equation} 
Next, denoting by $\widetilde{k}_{1,M} \in C^{\infty}(\mathbb{R}^d\times(\mathbb{R}^d\setminus\{0\}))$ the kernel associated to $K_{1,M}$ determined by Proposition \ref{Kernelpseudo}, by hypothesis \eqref{assumseparsupp} and integrating by parts we find
\begin{equation}\label{eqlem2.1}
\begin{aligned}
 \theta_2 (x)& K_{1,M}((1- \theta_1)f)(x) \\
%&=\theta_2(x) \int \widetilde{k}_{1,M}(x,x-y)(1- \theta_1(y))f(y)\, dy\\
&=\theta_2(x) \int \widetilde{k}_{1,M}(x,x-y)(1- \theta_1(y))J^{2m}J^{-2m}f(y)\, dy \\
&= \sum_{0\leq |\gamma_1|+|\gamma_2|\leq 2m} c_{\gamma_1,\gamma_2}\theta_2(x)\\
&\hspace{1cm}\times \int_{|x-y|\geq \delta} (\partial^{\gamma_1}\widetilde{k}_{1,M})(x,x-y)\partial^{\gamma_2}(1- \theta_1(y))J^{-2m}f(y)\, dy,
\end{aligned}
\end{equation}
for some constant $c_{\gamma_1,\gamma_2}$ with $0\leq |\gamma_1|+|\gamma_2|\leq 2m$, which are not relevant for our considerations. The preceding estimate, Proposition \ref{Kernelpseudo} (ii) for some integer $N>0$ fixed, and Young's inequality allow us to deduce
\begin{equation}\label{eqlem3}
\begin{aligned}
\|\theta_2 & K_{1,M}((1- \theta_1)f) \|_{L^2} \\
&\lesssim \|\theta_2\|_{L^{\infty}}\big(\sum_{0\leq |\gamma| \leq 2m}\|\partial^{\gamma}(1-\theta_1)\|_{L^{\infty}}\big)\sum_{l=0}^{2m} \Big\|\frac{\chi_{\{|\cdot|\geq \delta\}}}{|\cdot|^{d+l+N}}\ast |J^{-2m}f| \Big\|_{L^2}\\
& \lesssim_{\theta_1,\theta_2} \|J^{-m}f\|_{L^2},
\end{aligned}
\end{equation}
where we have also employed $\|J^{-2m}f\|_{L^2}\leq \|J^{-m}f\|_{L^2}$.  Collecting 
\eqref{eqlem2} and \eqref{eqlem3}, we conclude
\begin{equation*}
\|\theta_2 K_{1,M}f \|_{L^2} \lesssim \|\theta_1 f\|_{L^2}+ \|J^{-m}f\|_{L^2}.
\end{equation*}
On the other hand, since $K_{2,M}$ does not determine a pseudo-differential operator, we must employ a different reasoning to bound this operator. Instead, we write
\begin{equation*}
\begin{aligned}
K_{2,M}f=\sum_{j=M}^{\infty} \binom{\beta/2}{j} (-1)^{j+1} G_{2j-\beta}\ast f,
\end{aligned}
\end{equation*}
where $G_{\delta}$, $\delta>0$, denotes the Bessel kernel (see \cite{AronszBessel,steinSingular}) defined by
\begin{equation*}
\begin{aligned}
G_{\delta}(x)=c_{\delta} \int_0^{\infty} e^{-\pi |x|^2/w}e^{-w/4 \pi}w^{(-d+\delta)/2}\, \frac{dw}{w},
\end{aligned}
\end{equation*}
for some constant $c_{\delta}>0$. Additionally, we recall the estimate
\begin{equation}\label{eqlem4}
|\partial^{\gamma}G_{2}(x)|\leq c_{2}\big(G_{2}(x)+G_{1}(x)\big),
\end{equation}
which holds for all multi-index $\gamma$ of order $|\gamma|=1$. Now, writing $f=J^{-2m}J^{2m}f$, by properties between convolution and derivatives, it is not difficult to deduce
\begin{equation}\label{eqlem5}
\begin{aligned}
K_{2,M}f=\sum_{j=M}^{\infty} \binom{\beta/2}{j} (-1)^{j+1} \sum_{0\leq |\gamma|\leq 2m} c_{\gamma}\big((\partial^{\gamma}G_{2j-\beta})\ast J^{-2m}f\big).
\end{aligned}
\end{equation}
To estimate the above equality, we decompose each multi-index $\gamma$ with $0\leq |\gamma|\leq 2m$ as a sum of $2m$ multi-indexes of order less or equal than $1$, that is, $\gamma=\sum_{l=1}^{2m}\gamma_{l}$, where $0\leq |\gamma_l| \leq 1$. From this, we have
\begin{equation*}
\partial^{\gamma}G_{2j-\beta}=G_{2j-4m-\beta}\ast \underbrace{ \partial^{\gamma_1}G_2\ast \cdots \ast \partial^{\gamma_{2m}}G_2}_{2m-\text{times}},
\end{equation*}
for each $j\geq M> 2m-\beta/2$, and $0\leq |\gamma| \leq 2m$. Then, for these set of indexes, \eqref{eqlem4}, and the fact that $\|G_{\delta}\|_{L^1}=1$, for all $\delta>0$, imply
\begin{equation*}
\begin{aligned}
\|\partial^{\gamma}G_{2j-\beta} \ast J^{-2m}f\|_{L^{2}}\leq (2c_2)^{2m}\|J^{-2m}f\|_{L^2}.
\end{aligned}
\end{equation*}
Plugging the previous estimate in \eqref{eqlem5} reveals
\begin{equation}\label{eqlem5.1}
\|K_{2,M}f\|_{L^2}\lesssim \sum_{j=0}^{\infty} \left|\binom{\beta/2}{j}\right| \|J^{-2m}f\|_{L^2} \lesssim \|J^{-m}f\|_{L^2}.
\end{equation}
In particular, this shows $\theta_2 K_{2,M}f  \in L^{2}(\mathbb{R}^d),$ and in consequence the proof of (I) is complete.

Next, we deduce (II). %Since the case $\beta=0$ is clearly true, we assume that $0<\beta <2$. 
Writing $f=J^{-\beta}(J^{\beta} f)$, we have 
\begin{equation}\label{eqlem6}
\theta_2 f= \theta_2 J^{-\beta}(J^{\beta}f) =\theta_2 J^{-\beta}(\theta_1 J^{\beta}f) +\theta_2 J^{-\beta}((1-\theta_1)J^{\beta}f).
\end{equation}
The first term of the above equality satisfies
\begin{equation}\label{eqlem7}
\|\theta_2 J^{-\beta}(\theta_1 J^{\beta}f)\|_{L^2}\lesssim \|\theta_2\|_{L^{\infty}}\|\theta_1 J^{\beta}f\|_{L^2}.
\end{equation}
Now, the remaining estimate for the r.h.s of \eqref{eqlem6} is obtained by arguing exactly as in \eqref{eqlem2.1}. Indeed, letting $\widetilde{m}$ be an integer such that $2\widetilde{m}\geq m+ \beta$, and  $q_{\beta}\in C^{\infty}(\mathbb{R}^d\times(\mathbb{R}^d\setminus\{0\}))$ be the kernel associated to $J^{-\beta}$, by (ii) in Proposition \ref{Kernelpseudo} for some integer $N>0$, by using \eqref{assumseparsupp} and integrating by parts, we find
\begin{equation}\label{eqlem8}
\begin{aligned}
\|\theta_2 & J^{-\beta}((1-\theta_1)J^{\beta}f) \|_{L^2} \\
&\lesssim \sum_{0\leq |\gamma_1|+|\gamma_2|\leq 2\widetilde{m}}\|\theta_2(x)\big(\partial^{\gamma_1}q_{\beta}(x,\cdot)\ast(\partial^{\gamma_2}(1-\theta_1)J^{-2\widetilde{m}+\beta}f)\big)(x)\|_{L^2} \\
&\lesssim \|\theta_2\|_{L^{\infty}}\big(\sum_{0\leq |\gamma| \leq 2m}\|\partial^{\gamma}(1-\theta_1)\|_{L^{\infty}}\big)\sum_{l=0}^{2m} \Big\|\frac{\chi_{\{|\cdot|\geq \delta\}}}{|\cdot|^{d+l+N}}\ast |J^{-2\widetilde{m}+\beta}f| \Big\|_{L^2}\\
& \lesssim_{\theta_1,\theta_2} \|J^{-m}f\|_{L^2}.
\end{aligned}
\end{equation}
Gathering \eqref{eqlem6}-\eqref{eqlem8}, we conclude that $\theta_2 f\in L^{2}(\mathbb{R}^d)$.

On the other hand, following the same arguments in the proof of (i), we have
\begin{equation}\label{eqlem8.1}
\theta_2 D^{\beta}f=\theta_2 J^{\beta} f-\theta_2 K_{1,M}f-\theta_2 K_{2,M}f,
\end{equation}
where $M>1$ is a fixed integer number such that $M> 2m+\beta/2$, and the operators $K_{j,M}$ are defined as above according to $k_{j,m}$ given by \eqref{eqlem1} for all $j=1,2$. Notice that \eqref{eqlem5.1} establishes the desired estimate for $\theta_2 K_{2,M}f $.

To control $\theta_2 K_{1,M}f$, once again we set $J^{-\beta}J^{\beta}f$, then denoting by $\widetilde{K}_{1,M}$ the pseudo-differential operator given by the composition $K_{1,M} J^{-\beta}$ (see Proposition \ref{PO1}), it is seen that
\begin{equation*}
\begin{aligned}
\theta_2 K_{1,M}f=\theta_2 K_{1,M} J^{-\beta} (J^{\beta}f)=\theta_2 \widetilde{K}_{1,M}(\theta_1 J^{\beta}f) +\theta_2\widetilde{K}_{1,M}((1-\theta_1)J^{\beta}f).
\end{aligned}
\end{equation*}
Consequently, the previous equality is bounded by the same estimates concerning the  r.h.s of \eqref{eqlem6}. To avoid repetitions, we omit the details. From this, it follows $\theta_2 K_{1,M}f\in L^{2}(\mathbb{R}^2)$, and so, by equation \eqref{eqlem8.1}, $\theta_2 D^{\beta}f  \in L^{2}(\mathbb{R}^2)$ which establishes (II).

To deduce (III), we decompose
\begin{equation*}
\theta_2 J^r f=\theta_2 J^{-(s-r)}J^sf=\theta_2 J^{-(s-r)}(\theta_1 J^sf)+\theta_2J^{-(s-r)}((1-\theta_1)J^sf).
\end{equation*}
The above identity and similar considerations as in \eqref{eqlem6} yield the deduction of (III).

Finally, we deal with (IV). Recalling that $f\in H^{-m}(\mathbb{R}^d)$, we consider an integer $m_1>\max\{s+m,d\}$ such that by Proposition \ref{PO3} it is seen that
\begin{equation}\label{eqlem9}
J^{s}(f\theta_2)=\sum_{1\leq |\gamma|\leq m_1} c_{\gamma}\partial_{x}^{\gamma} \theta_2 \, \Psi^{\gamma}(J^{s-|\gamma|}f)+\theta_2 J^{s}f +K_{s-m_1}f,
\end{equation}
where $\Psi^{\gamma}$ is a given pseudo-differential operator of order zero for each $1\leq |\gamma|\leq m_1$ and $K_{s-m_1}$ is of order $s-m_1$. Clearly, $\|\theta_2 J^{s}f\|_{L^2}\lesssim \|\theta_1 J^{s}f\|_{L^2}$, thus we focus on the remaining parts in \eqref{eqlem9}. We first estimate
\begin{equation*}
\begin{aligned}
\|K_{s-m_1}f\|_{L^2}=\|K_{s-m_1}J^{m}J^{-m}f\|_{L^2}\lesssim \|J^{-m}f\|_{L^2}.
\end{aligned}
\end{equation*}
Now, for each multi-index $1\leq |\gamma|\leq m_1$, we write
\begin{equation*}
\begin{aligned}
\partial_{x}^{\gamma} \theta_2 \, \Psi^{\gamma}(J^{s-|\gamma|}f)=\partial_{x}^{\gamma} \theta_2 \, \big(\Psi^{\gamma}J^{-|\gamma|}\big)(\theta_1 J^{s}f)+\partial_{x}^{\gamma} \theta_2\big(\Psi^{\gamma}J^{-|\gamma|}\big)((1-\theta_1) J^{s}f).
\end{aligned}
\end{equation*}
By recurrent arguments using that $\Psi^{\gamma}J^{-|\gamma|}$ is a pseudo-differential operator of order zero and the assumption on the supports, on one hand we have
\begin{equation*}
\begin{aligned}
\|\partial_{x}^{\gamma} \theta_2\big(\Psi^{\gamma}J^{-|\gamma|}\big)((1-\theta_1) J^{s}f)\|_{L^2}\lesssim \|J^{-m}f\|_{L^2},
\end{aligned}
\end{equation*}
while on the other it is seen that
\begin{equation*}
\begin{aligned}
\|\partial_{x}^{\gamma} \theta_2 \, \big(\Psi^{\gamma}J^{-|\gamma|}\big)(\theta_1 J^{s}f)\|_{L^2}\lesssim \|\theta_1 J^{s}f\|_{L^2}.
\end{aligned}
\end{equation*}
Gathering the previous results we complete the deduction of (IV).
\end{proof}

%%%%%%%%%%%%%%%%%%%%%%%%%%%%%%%%%%%%%%%%%%%%%%%%%%%%%%%%%%%%%%%%%%%%%%%%%%%%%%%%%%%%%%%%%%%%%%%%%%%%%%%%%%%%%%%%%%%%%%%%%%%%%%%%%%%%%%%%%%%%%%%%%%%%%

\subsection{Weighted functions}\label{subweigh}

In this part, we introduce the cutoff functions to be employed in our arguments. This class of functions was first used in \cite{IsazLP,keLinaPOnV}. For the sake of brevity, we will only present those properties required for our considerations. For a more detailed discussion, see Isaza, Linares and Ponce \cite{IsazLP}. 

Given $\epsilon>0$ and $b\geq 5\epsilon$, we define the family of functions
\begin{equation*}
    \chi_{\epsilon,b}, \phi_{\epsilon,b},\tilde{\phi}_{\epsilon,b}, \psi_{\epsilon}, \eta_{\epsilon,b}\in C^{\infty}(\mathbb{R}),
\end{equation*}
satisfying the following properties:
\begin{itemize}
    \item[(i)] $\chi_{\epsilon,b}' \geq 0$,
    \item[(ii)] 
    $\chi_{\epsilon,b}(x)=\left\{\begin{aligned}
     0, \, \, x\leq \epsilon \\ 1, \, \, x\geq b,
    \end{aligned}
     \right.$,
    \item[(iii)]\label{chi'property} $\chi_{\epsilon,b}'(x) \geq \frac{1}{10(b-\epsilon)}\mathbbm{1}_{[2\epsilon,b-2\epsilon]}(x)$,
    \item[(iv)]\label{lowerboudnchi} $\chi_{\epsilon,b}(x) \geq \frac{1}{2} \frac{\epsilon}{b-3\epsilon}$, whenever $x\in (3\epsilon,\infty)$,
    \item[(v)] $\supp(\chi'_{\epsilon,b})\subset [\epsilon,b]$,
    \item[(vi)] \label{phiproperty} $\supp(\phi_{\epsilon,b}),\supp(\widetilde{\phi}_{\epsilon,b})\subset [\epsilon/4,b]$,
    \item[(vii)] $\phi_{\epsilon}(x)=\widetilde{\phi}_{\epsilon,b}(x)=1, x\in [\epsilon/2,\epsilon]$,
    \item[(viii)] $\supp(\psi_{\epsilon})\subset (-\infty,\epsilon/2]$.
    \item[(ix)]  Given $x\in \mathbb{R}$, we have the following partitions of unity
    \begin{equation}\label{phidecomp}
        \chi_{\epsilon,b}(x)+\phi_{\epsilon,b}(x)+\psi_{\epsilon}(x)=1
    \end{equation}
    and
    \begin{equation}\label{phidecomp2}
        \chi_{\epsilon,b}^2(x)+\widetilde{\phi_{\epsilon,b}}^{2}(x)+\psi_{\epsilon}(x)=1.
    \end{equation}
\end{itemize}
By a slight abuse of notation, when it is required, we shall assume that the above functions act in two variables as follows $\chi_{\epsilon,b}(x,y):=\chi_{\epsilon,b}(x)$, similarly for the other weighted functions introduced above.

%%%%%%%%%%%%%%%%%%%%%%%%%%%%%%%%%%%%%%%%%%%%%%%%%%%%%%%%%%%%%%%%%%%%%%%%%%%%%%%%%%%%%%%%%%%%%%%%%%%%%%%%%%%%%%%%%%%%%%%%%%%%%%%%%%%%%%%%%%%%%%%%%%%%%%%%

\section{Kato's smoothing effect}\label{Katosect} 

We are in the condition to establish the following Kato's smoothing effect for solutions of \eqref{BOZK}. 

\begin{prop}\label{KATOSM}
Let $0\leq \alpha < 1$. Consider $s>s_{\alpha}=(17-2\alpha)/12$, and $u_0\in H^s(\mathbb{R}^2)$. Then the corresponding solution $u\in C([0,T]; H^{s}(\mathbb{R}^2))$ of the IVP \eqref{BOZK} with initial data $u_0$ determined by Theorem \ref{imprwellpos} satisfies for any $R>0$, $T>0$ and $0\leq r \leq s$ that
\begin{equation}\label{eqsmooth} 
    D_x^{\frac{\alpha+1}{2}}A^r u, \, \,  \mathcal{H}_x D_x^{\frac{\alpha+1}{2}}A^r u, \, \, \partial_y A^r u \, \in L^{2}((-R,R)_x\times \mathbb{R}_y \times (0,T)),
\end{equation}
where $A^r$ is any among the operators $J^r, J_x^r, J_y^r, D^r, D_x^r$ and $D^r_y$. 
\end{prop}

\begin{proof} 
We first consider the case $A^r=J^r$ for fixed $0\leq r \leq s$. The following computations can be justified approximating with smooth solutions of \eqref{BOZK} and taking the limit in our estimates. Thus, we will perform our considerations assuming the required regularity on the solution. We let $\psi \in C^{\infty}(\mathbb{R})$ with $\psi'\geq 0$, and $\psi'$ compact supported. Applying the operator $J^r$ to \eqref{BOZK}, then multiplying the resulting expression by $J^r u \psi$ and integrating in space, we find
\begin{equation}\label{eqpre1}
\begin{aligned}
\frac{1}{2}\frac{d}{dt}\int |J^r u|^2 &\psi(x) \, dx dy-\int (\partial_x D_x^{\alpha+1} J^r u) J^r u \psi(x) \, dx dy \\
&+\frac{1}{2}\int|\partial_y J^r u|^2\psi'(x) \, dx dy +\int J^r(u u_x)J^r u \psi(x) \, dx dy=0.
\end{aligned}
\end{equation}
To deal with the second term on the left-hand side of \eqref{eqpre1}, by writing $\partial_x D_x^{\alpha+1}=-\mathcal{H}_x D_x^{\alpha+2}$, we apply the expansion \eqref{commutremind} with $a=\alpha+2$, $b=0$ and $n=0$ to deduce
\begin{equation}\label{equsmooth1}
    \begin{aligned}
    &-\int (\partial_x D_x^{\alpha+1} J^r u) J^r u \psi \, dx dy \\
    &\quad =\frac{1}{2}\int J^r u[-\mathcal{H}_x D_x^{\alpha+2},\psi]J^r u  \, dx dy \\
    &\quad =\frac{(\alpha+2)}{4}\int |D_x^{\frac{\alpha+1}{2}}J^r u|^2 \psi' \, dx dy+\frac{(\alpha+2)}{4}\int |\mathcal{H}_x D_x^{\frac{\alpha+1}{2}}J^r u|^2 \psi' \, dx dy\\ 
    &\qquad+\frac{1}{2} \int J^r u R_0(\alpha+2)J^r u  \, dx dy,
    \end{aligned}
\end{equation}
where $\left\|R_0(\alpha+2)\right\|_{L^2_{xy}\to L^2_{xy}}\leq c=c(\psi')$.

On the other hand, to bound the third term on the left-hand side of \eqref{eqpre1}, we integrate by parts to obtain
\begin{equation*}
    \begin{aligned}
    &\int J^r(u u_x)J^r u \psi \, dx dy \\
    &\hspace{0.5cm}=\int [J^r, u]u_x J^ru \psi \, dx dy+\frac{1}{2}\int u \partial_x| J^r u|^2\psi \, dx dy \\
    &\hspace{0.5cm}=\int [J^r, u]u_x J^ru \psi \, dx dy-\frac{1}{2}\int u_x | J^r u|^2\psi \, dx dy-\frac{1}{2}\int u | J^r u|^2\psi' \, dx dy.
    \end{aligned}
\end{equation*}
Consequently, the preceding equality, the fact that $r\leq s$ and the Kato-Ponce inequality \eqref{conmKPI} allow us to deduce
\begin{equation}\label{eqpre2}
    \begin{aligned}
    &\left|\int J^r(u u_x)J^r u \psi \, dx dy\right| \\
    &\hspace{0.5cm}\lesssim \left\|[J^r,u]u_x\right\|_{L^2_{xy}}\left\|J^ru\right\|_{L^2_{xy}}+(\left\|u\right\|_{L^{\infty}_{xy}}+\left\|\partial_x u\right\|_{L^{\infty}_{xy}})\left\|J^r u\right\|_{L^2_{xy}}^2 \\
    &\hspace{0.5cm}\lesssim (\left\|u\right\|_{L^{\infty}_{xy}}+\left\|\nabla  u\right\|_{L^{\infty}_{xy}})\left\|J^s u\right\|_{L^2_{xy}}^2.
    \end{aligned}
\end{equation}
We remark that the implicit constant above depends on $\left\|\psi\right\|_{L^\infty}$ and $\left\|\psi'\right\|_{L^{\infty}}$. Thus, gathering the above estimates yields
\begin{equation}\label{eqpre3}
\begin{aligned}
\frac{1}{2}\frac{d}{dt}\int |J^r u|^2 \psi \, dx dy&+\frac{(\alpha+2)}{4}\int |D_x^{\frac{\alpha+1}{2}}J^r u|^2 \psi' \, dx dy \\
&+\frac{(\alpha+2)}{4}\int |\mathcal{H}_x D_x^{\frac{\alpha+1}{2}}J^r u|^2 \psi' \, dx dy \\
&+\frac{1}{2}\int|\partial_y J^r u|^2\psi' \, dx dy  \\
&\leq (1+\left\|u\right\|_{L^{\infty}_{xy}}+\left\|\nabla  u\right\|_{L^{\infty}_{xy}})\left\|J^s u\right\|_{L^2_{xy}}^2.
\end{aligned}
\end{equation}
Noticing that Theorem \ref{imprwellpos} establishes that $u\in L^{1}([0,T];W^{1,\infty}(\mathbb{R}^2))$, we can apply Gronwall's lemma in \eqref{eqpre3}, obtaining the desired conclusion for the case $A^r=J^r$. 

The estimates for the remaining cases $A^r$ follow by similar reasoning as above, the only modification required concerns the estimate for the nonlinear term, i.e., \eqref{eqpre3}. By implementing \eqref{conmKPI} on each variable, we still control the cases $A^r=J^r_x,J^r_y$ (the former cases holds by the assumption $\partial_y u\in L^{1}([0,T];L^{\infty}(\mathbb{R}^2))$). Whereas, Lemma \ref{nonhest}  allows us to deal with  $A^r=D^r, D_x^r,D_y^r$. The proof is complete.
\end{proof}
%%%%%%%%%%%%%%%%%%%%%%%%%%%%%%%%%%%%%%%%%%%%%%%%%%%%%%%%%%%%%%%%%%%%%%%%%%%%%%%%%%%%%%%%%%%%%%%%%%%%%%%%%%%%%%%%%%%%%%%%%%%%%%%%%%%%%%%%%%%%%%%%%%%%%%%%%%%%%%%%%%%%%%%%%%%%%%%%%%%%%%%%%%%
\section{Proof of Theorem \ref{mainTheo}}\label{mainTheosection}

Our analysis follows the technique introduced in \cite{IsazLP} (see also \cite{keLinaPOnV,2019P1Argenis,2019P2Argenis,ArgenisZK}), so that our starting point will be basically to obtain weighted energy estimate by localizing the regions in $\mathbb{R}^{2}$ where the information concerning the regularity is available. 

By translation, we may set $x_0=0$. Additionally, we shall assume that the solution $u$ of the IVP \eqref{BOZK} % with initial data $u_0$ satisfying \eqref{hypomain} 
 has the required regularity to justify our estimates. At the end, the desired conclusion follows by a limit process employing smooth solutions and our estimates. Therefore, by applying directly the operator $J^{\bar{s}}_x$ to the equation in \eqref{BOZK}, followed by a multiplication by $J^{\bar{s}}_xu(x,y) \chi_{\epsilon,b}^{2}(x+vt),$ that combined with integration by parts allow us to deduce the identity
\begin{equation}\label{identnonl}
    \begin{split} 
    &\frac{1}{2}\frac{d}{dt}\int (J^{\bar{s}}_x u)^{2}\chi_{\epsilon,b}^{2}\, dx dy \,\underbrace{-\frac{v}{2}\int (J^{\bar{s}}_x u)^2(\chi_{\epsilon,b}^{2})'dxdy }_{A_{1}(t)}  \\
    &\underbrace{-\int(\partial_{x} D_{x}^{1+\alpha}J^{\bar{s}}_xu)J^{\bar{s}}_x u\chi_{\epsilon,b}^{2}\, dxdy}_{A_{2}(t)}+\int \partial_x\partial_{y}^2J^{\bar{s}}_xu J^{\bar{s}}_xu \chi_{\epsilon,b}^2\, dx dy 
   \\
            &\underbrace{+\int J^{\bar{s}}_x(u\partial_{x}u)J^{\bar{s}}_x u\chi_{\epsilon,b}^{2}\, dx dy}_{A_{3}(t)} = 0.
    \end{split}
 \end{equation}
We notice that integrating by part yields
\begin{equation*}
\begin{aligned}
\int \partial_x\partial_{y}^2J^{\bar{s}}_xu J^{\bar{s}}_xu \chi_{\epsilon,b}^2\, dx dy=\int (\partial_{y}J^{\bar{s}}_xu)^2\chi_{\epsilon,b}\chi_{\epsilon,b}'\, dx dy \geq 0.
\end{aligned}
\end{equation*}
Consequently, we only need to estimate $A_1(t)$, $A_2(t)$ and $A_3(t)$ determined by \eqref{identnonl}. To simplify the exposition, the preceding differential inequality and the corresponding terms $A_j(t)$, $j=1,2,3$ will be employed for different values $\bar{s}$ previously fixed. We notice that our objective is bounding equation \eqref{identnonl} corresponding to the case $\bar{s}=s$, whenever $s>s_{\alpha}$. 

Since the estimate for $A_2(t)$ follows by rather general arguments independent of $\bar{s}$, for the sake of brevity, we develop this estimate in the following lemma. 

\begin{lemma}\label{estimatA2} 
Let $\chi_{\epsilon,b},\phi_{\epsilon,b}$ and $\psi_{\epsilon}$ defined as in Subsection \ref{subweigh}, satisfying \eqref{phidecomp}. Then there exist some positive constants $c_0, c_1$ such that
\begin{equation*}
\begin{aligned}
A_2(t)=Sm_1(t)+Sm_2(t)+R_{A_2}(t),
\end{aligned}
\end{equation*}
where
\begin{equation*} 
\begin{aligned}
&Sm_1(t)=\frac{\alpha+2}{2}\big(\|(\chi_{\epsilon,b}\chi_{\epsilon,b}')^{1/2}D_x^{\frac{\alpha+1}{2}}J^{\bar{s}}_x u(t)\|_{L^2_{xy}}^2 \\
&\hspace{4cm}+ \|(\chi_{\epsilon,b}\chi_{\epsilon,b}')^{1/2}\mathcal{H}_xD_x^{\frac{\alpha+1}{2}}J^{\bar{s}}_x u(t)\|_{L^2_{xy}}^2 \big),\\
&Sm_2(t):= \frac{\alpha+2}{2}\big(\|(\chi_{\epsilon,b}\chi_{\epsilon,b}')^{1/2}D_x^{\frac{\alpha+1}{2}}J^{\bar{s}}_x(u\psi_{\epsilon})(t)\|_{L^2_{xy}}^2 \\
&\hspace{4cm}+ \|(\chi_{\epsilon,b}\chi_{\epsilon,b}')^{1/2}\mathcal{H}_xD_x^{\frac{\alpha+1}{2}}J^{\bar{s}}_x (u\psi_{\epsilon})(t)\|_{L^2_{xy}}^2 \big).
\end{aligned}
\end{equation*}
and
\begin{equation}\label{eqlemma1} 
\begin{aligned}
|R_{A_2}(t)|\leq & \frac{1}{4}Sm_1(t)+c_0\|u\|_{H^{s_{\alpha}^{+}}}^2 \\
&+ c_1\sum_{0\leq j \leq \max\{\bar{s}-s_{\alpha},0\}} \|\chi_{\epsilon,b}J^{\bar{s}-j}_xu\|_{L^2_{xy}}^2+\|J^{\bar{s}-j}_x(u \phi_{\epsilon,b})\|_{L^2_{xy}}^2. 
\end{aligned}
\end{equation}
\end{lemma}

\begin{proof}
By writing $\partial_x=-\mathcal{H}_xD_x$  and using \eqref{phidecomp}, we decompose the estimate for $A_2$ as follows
\begin{equation}\label{eqpropa0}
\begin{split}
    A_2(t) =& \frac{1}{2}\int_{\mathbb{R}^{2}} J^{\bar{s}}_xu\left[D_x^{\alpha + 1}\partial_x, \chi_{\epsilon,b}^2\right]J^{\bar{s}}_xu\, dx dy \\
     =& \frac{1}{2}\int_{\mathbb{R}^{2}} J^{\bar{s}}_x\left( (u\chi_{\epsilon,b}) + (u\phi_{\epsilon,b}) + (u\psi_{\epsilon}) \right)\\
     &\hspace{1cm}\times \left[-\mathcal{H}_xD_x^{\alpha + 2}, \chi_{\epsilon,b}^2\right]J^{\bar{s}}_x\left( (u\chi_{\epsilon,b}) + (u\phi_{\epsilon,b}) + (u\psi_{\epsilon}) \right)\,dx\,dy \\
=&\sigma(J^{\bar{s}}_x( u\chi_{\epsilon,b}),J^{\bar{s}}_x(u\chi_{\epsilon,b}))+2\sigma(J^{\bar{s}}_x(u\chi_{\epsilon,b}),J^{\bar{s}}_x(u\phi_{\epsilon,b}))\\
&+\sigma(J^{\bar{s}}_x(u\phi_{\epsilon,b}),J^{\bar{s}}_x(u\phi_{\epsilon,b}))+2\sigma(J^{\bar{s}}_x(u\chi_{\epsilon,b}),J^{\bar{s}}_x(u\psi_{\epsilon}))\\
&+2\sigma(J^{\bar{s}}_x(u\phi_{\epsilon,b}),J^{\bar{s}}_x( u\psi_{\epsilon}))+\sigma(J^{\bar{s}}_x(u\psi_{\epsilon}),J^{\bar{s}}_x( u\psi_{\epsilon})) \\
=:&A_{2,1}(t)+A_{2,2}(t)+A_{2,3}(t)+A_{2,4}(t)+A_{2,5}(t)+A_{2,6}(t),
\end{split}
\end{equation}
where we have set 
\begin{equation*}
\sigma(f,g)=\frac{1}{2}\int f\left[-\mathcal{H}_xD_x^{\alpha + 2}, \chi_{\epsilon,b}^2\right]g \, dx dy,
\end{equation*}
and we have used
\begin{equation*}
\sigma(f,g)=\sigma(g,f).
\end{equation*}
To deduce a convenient factorization for the operator $\sigma$, we employ Proposition \ref{valueofn} with $a=\alpha+2$, $b=0$ and $n=0$ to get 
\begin{equation}\label{eqpropa1}
\begin{aligned}
\sigma(f,g)=&\frac{\alpha+2}{2}\int (D_x^{\frac{\alpha+1}{2}}f)(D_x^{\frac{\alpha+1}{2}}g)\chi_{\epsilon,b}\chi'_{\epsilon,b}\, dxdy\\
&+\frac{\alpha+2}{2}\int (\mathcal{H}_xD_x^{\frac{\alpha+1}{2}}f)(\mathcal{H}_xD_x^{\frac{\alpha+1}{2}}g)\chi_{\epsilon,b}\chi'_{\epsilon,b}\, dxdy \\
&+\frac{1}{2}\int f R_0(\alpha+2)g\, dxdy,
\end{aligned}
\end{equation}
where the operator $R_0(\alpha+2)$ satisfies
\begin{equation}\label{boundGV}
\|R_0(\alpha+2)\|_{L^2_{xy}\rightarrow L^2_{xy}} \lesssim \|\widehat{D^{\alpha+2}_x(\chi_{\epsilon,b}^2)}(\xi)\|_{L^1_{\xi}}.
\end{equation}
Consequently, we divide the analysis of \eqref{eqpropa0} according to those cases where the above decomposition leads bounded expressions, and in those where we can estimate directly the commutator defining $\sigma(\cdot,\cdot)$. Indeed, \eqref{eqpropa1}, and the fact that $\chi_{\epsilon,b}+\phi_{\epsilon,b}=1-\psi_{\epsilon}$ allow us to write
\begin{equation*}
\begin{aligned}
A_{2,1}&(t)+A_{2,2}(t)+A_{2,3}(t)\\
%=&{\color{blue}\frac{\alpha+2}{2}\int \big(J^{\bar{s}}_xD_x^{\frac{\alpha+1}{2}}(u\chi_{\epsilon,b})+J^{\bar{s}}_xD_x^{\frac{\alpha+1}{2}}(u\phi_{\epsilon,b})\big)^2 \chi_{\epsilon,b}\chi_{\epsilon,b}'\, dx dy} \\
%&{\color{blue}+\frac{\alpha+2}{2} \int \big(J^{\bar{s}}_x\mathcal{H}_xD_x^{\frac{\alpha+1}{2}}(u\chi_{\epsilon,b})+J^{\bar{s}}_x\mathcal{H}_xD_x^{\frac{\alpha+1}{2}}(u\phi_{\epsilon,b})\big)^2  \chi_{\epsilon,b}\chi_{\epsilon,b}'\, dx dy}\\
%&{\color{blue}+ \int \big(J^{\bar{s}}_x(u\chi_{\epsilon,b})+J^{\bar{s}}_x(u\phi_{\epsilon,b})\big)R_{0}(\alpha+2)\big(J^{\bar{s}}_c(u\chi_{\epsilon,b})+J^{\bar{s}}_x(u\phi_{\epsilon,b})\big)} \\
=&\frac{\alpha+2}{2} \int \Big( \big(D_x^{\frac{\alpha+1}{2}}J^{\bar{s}}_x u\big)^2  +  \big(\mathcal{H}_xD_x^{\frac{\alpha+1}{2}}J^{\bar{s}}_x u\big)^2 \Big)  \chi_{\epsilon,b}\chi_{\epsilon,b}'\, dx dy \\
& +\frac{\alpha+2}{2} \int \Big( \big( D_x^{\frac{\alpha+1}{2}}J^{\bar{s}}_x  (u\psi_{\epsilon})\big)^2  +  \big( \mathcal{H}_xD_x^{\frac{\alpha+1}{2}}J^{\bar{s}}_x(u\psi_{\epsilon})\big)^2 \Big)  \chi_{\epsilon,b}\chi_{\epsilon,b}'\, dx dy \\
&-(\alpha+2)\int \Big( (D_x^{\frac{\alpha+1}{2}}J^{\bar{s}}_x u)(D_x^{\frac{\alpha+1}{2}}J^{\bar{s}}_x(u\psi_{\epsilon}))\\
& \hspace{2cm}+(\mathcal{H}_xD_x^{\frac{\alpha+1}{2}}J^{\bar{s}}_x u)(\mathcal{H}_xD_x^{\frac{\alpha+1}{2}}J^{\bar{s}}_x(u\psi_{\epsilon}))  \Big)  \chi_{\epsilon,b}\chi_{\epsilon,b}'\, dx dy \\
&+ \frac{1}{2}\int \big(J^{\bar{s}}_x(u\chi_{\epsilon,b})+J^{\bar{s}}_x(u\phi_{\epsilon,b})\big)R_{0}(\alpha+2)\big(J^{\bar{s}}_x(u\chi_{\epsilon,b})+J^{\bar{s}}_x(u\phi_{\epsilon,b})\big)\\
=:&Sm_{1}(t)+Sm_2(t)+\widetilde{A}_{2,1}(t)+\widetilde{A}_{2,2}(t).
\end{aligned}
\end{equation*}
Integrating between $[0,T]$, we notice that $Sm_{1}(t)$ corresponds to the required smoothing effect, while $Sm_{2}(t)$ provides a positive quantity. On the other hand, since $R_0(\alpha+2)$ satisfies \eqref{boundGV}, we have
\begin{equation}\label{eqmainT1}
\begin{aligned}
|\widetilde{A}_{2,2}(t)| &\lesssim \|J^{\bar{s}}(u\chi_{\epsilon,b})\|_{L^2_{xy}}^2+\|J^{\bar{s}}(u\phi_{\epsilon,b})\|_{L^2_{xy}}^2\\
&=:\widetilde{A}_{2,2,1}(t)+\|J^{\bar{s}}(u\phi_{\epsilon,b})\|_{L^2_{xy}}^2.
\end{aligned}
\end{equation}
Let $m_1\geq \max\{2,\bar{s}\}$, by employing Proposition \ref{PO3} with $\theta(x)=\chi_{\epsilon,b}(x+v t)$, it is seen that
\begin{equation}\label{eqmainT1.1}
\begin{aligned}
|\widetilde{A}_{2,2,1}(t)|\lesssim & \sum_{j=1}^{m_1}\|\|\chi_{\epsilon,b}^{j}\Psi^{(j)}(J^{\bar{s}-j}_xu)\|_{L^2_x}\|_{L^2_{y}}^2+\|\|\chi_{\epsilon,b}J^{\bar{s}}_xu\|_{L^2_x}\|_{L^2_{y}}^2\\
&+\|\|u\|_{L^2_x}\|_{L^2_{y}}^2 \\
\lesssim &  \sum_{j=1}^{m_1}\|\chi_{\epsilon,b}^{j}\Psi^{(j)}(J^{\bar{s}-j}_xu)\|_{L^2_{xy}}^2+\|\chi_{\epsilon,b}J^{\bar{s}}_xu\|_{L^2_{xy}}^2+\|u\|_{L^2_{xy}}^2.
\end{aligned}
\end{equation}
Now, for each $j=1,\dots,m_1$, we employ \eqref{phidecomp} to deduce
\begin{equation}\label{eqmainT1.2}
\begin{aligned}
\|\chi_{\epsilon,b}^{j}\Psi^{(j)}&(J^{\bar{s}-j}_xu)\|_{L^2_{xy}}\\
\lesssim & 
\|\chi_{\epsilon,b}^{j}\Psi^{(j)}(J^{\bar{s}-j}_x(u \chi_{\epsilon,b}))\|_{L^2_{xy}}+\|\chi_{\epsilon,b}^{j}\Psi^{(j)}(J^{\bar{s}-j}_x(u\phi_{\epsilon,b}) )\|_{L^2_{xy}}\\
&+\|\chi_{\epsilon,b}^{j}\Psi^{(j)}(J^{\bar{s}-j}_x(u\psi_{\epsilon}))\|_{L^2_{xy}} \\
\lesssim & \|J^{\bar{s}-j}_x(u \chi_{\epsilon,b})\|_{L^2_{xy}}+\|J^{\bar{s}-j}_x(u \phi_{\epsilon,b})\|_{L^2_{xy}}+\|u\|_{L^2_{xy}},
\end{aligned}
\end{equation}
where we have also applied Lemma \ref{supportseparetedJ} to estimate
\begin{equation}\label{eqmainT1.3}
\begin{aligned}
\|\chi_{\epsilon,b}^{j}\Psi^{(j)}(J^{\bar{s}-j}_x(u\psi_{\epsilon}))\|_{L^2_{xy}}
%\|\|\chi_{\epsilon,b}^{j}\Psi^{(j)}J^{2m_1}_x(J^{-2m_1}J^{\bar{s}-j}_x(u\psi_{\epsilon}))\|_{L^2_x}\|_{L^2_{y}}\\
%&\lesssim \|\|J^{-2m_1}_xJ^{\bar{s}-j}_x(u\psi_{\epsilon})\|_{L^2_x}\|_{L^2_{y}}\\
& \lesssim \|u\|_{L^{2}_{xy}}. 
\end{aligned}
\end{equation}
We emphasize that the above considerations yield the factor $\|J^{\bar{s}-j}_x(u \chi_{\epsilon,b})\|_{L^2_{xy}}$. This suggests that we can iterate the arguments in \eqref{eqmainT1.1}-\eqref{eqmainT1.3} decreasing the derivatives considered in each step until we arrive at
\begin{equation*}
\begin{aligned}
|\widetilde{A}_{2,2,1}(t)| &\lesssim  \sum_{j=0}^{m_1} \|\chi_{\epsilon,b}J^{\bar{s}-j}_xu\|_{L^2_{xy}}^2+\|J^{\bar{s}-j}_x(u \phi_{\epsilon,b})\|_{L^2_{xy}}^2+\|u\|_{H^{s_{\alpha}^{+}}}^2.
\end{aligned}
\end{equation*}
Since $u\in C([0,T];H^{s_{\alpha}^{+}}(\mathbb{R}^2))$, we can modify the constant in the above inequality to reduce the previous sum to integers $0\leq j \leq \max\{\bar{s}-s_{\alpha},0\}$. Thus, we gather these conclusions to deduce
\begin{equation*}
\begin{aligned}
|\widetilde{A}_{2,2}(t)| &\lesssim \sum_{0\leq j \leq \max\{\bar{s}-s_{\alpha},0\}} \|\chi_{\epsilon,b}J^{\bar{s}-j}_xu\|_{L^2_{xy}}^2+\|J^{\bar{s}-j}_x(u \phi_{\epsilon,b})\|_{L^2_{xy}}^2+\|u\|_{H^{s_{\alpha}^{+}}}^2.
\end{aligned}
\end{equation*}
Now, we turn to $\widetilde{A}_{2,1}(t)$. By Young's inequality $a_1a_2\leq \frac{a_1^p}{p}+\frac{a_2^{p'}}{p'}$, with $\frac{1}{p}+\frac{1}{p'}=1$, and the fact that $\chi_{\epsilon,b}\chi_{\epsilon,b}'\geq 0$, we have
\begin{equation}\label{eqmainT2}
\begin{aligned}
|\widetilde{A}_{2,1}(t)|\leq & \frac{1}{4}|Sm_1(t)|+2(\alpha+2)\|(\chi_{\epsilon,b}\chi_{\epsilon,b}')^{1/2}D_x^{\frac{\alpha+1}{2}}J^{\bar{s}}_x(u\psi_{\epsilon})\|_{L^2_{xy}}^2\\
&+2(\alpha+2)\|(\chi_{\epsilon,b}\chi_{\epsilon,b}')^{1/2} \mathcal{H}_xD_x^{\frac{\alpha+1}{2}}J^{\bar{s}}_x(u\psi_{\epsilon})\|_{L^2_{xy}}^2.
\end{aligned}
\end{equation}
Bearing in mind that
\begin{equation*}
\dist(\supp(\chi_{\epsilon,b}\chi_{\epsilon,b}'),\supp(\psi_{\epsilon}))\geq \epsilon/2,
\end{equation*}
an application of Lemma \ref{supportsepareted2} on the $x$-spatial variable yields
\begin{equation}\label{eqmainT3}
\begin{aligned}
\|(\chi_{\epsilon,b}\chi_{\epsilon,b}')^{1/2} D_x^{\frac{\alpha+1}{2}}J^{\bar{s}}_x(u\psi_{\epsilon})\|_{L^2_{xy}}&\lesssim \|(\chi_{\epsilon,b}\chi_{\epsilon,b}')^{1/2}\|_{L^{\infty}}\|u\|_{L^2_{xy}}.
\end{aligned}
\end{equation}
To estimate the third term on the right-hand side of \eqref{eqmainT2}, we consider a function $\vartheta_{\epsilon,b}(x)=\vartheta_{\epsilon,b}(x+vt)\in C^{\infty}_c(\mathbb{R})$ with $0\leq \vartheta_{\epsilon,b} \leq 1$, such that $\vartheta_{\epsilon,b}\chi_{\epsilon,b}\chi_{\epsilon,b}'=\chi_{\epsilon,b}\chi_{\epsilon,b}'$ and $\dist(\supp(\vartheta_{\epsilon,b}),\supp(\psi_{\epsilon}))\geq \epsilon/8$, then we write
\begin{equation*}
\begin{aligned}
(\chi_{\epsilon,b}\chi_{\epsilon,b}')^{1/2} \mathcal{H}_xD_x^{\frac{\alpha+1}{2}}J^{\bar{s}}_x(u\psi_{\epsilon})=&(\chi_{\epsilon,b}\chi_{\epsilon,b}')^{1/2}\vartheta_{\epsilon,b} \mathcal{H}_xD_x^{\frac{\alpha+1}{2}}J^{\bar{s}}_x(u\psi_{\epsilon})\\
=&(\chi_{\epsilon,b}\chi_{\epsilon,b}')^{1/2}[\vartheta_{\epsilon,b},\mathcal{H}_x]D_x^{\frac{\alpha+1}{2}}J^{\bar{s}}_x(u\psi_{\epsilon})\\
&+(\chi_{\epsilon,b}\chi_{\epsilon,b}')^{1/2}\mathcal{H}_x\big(\vartheta_{\epsilon,b} D_x^{\frac{\alpha+1}{2}}J^{\bar{s}}_x(u\psi_{\epsilon}) \big).
\end{aligned}
\end{equation*}
Clearly, since $\mathcal{H}_x$ is a bounded operator on $L^2(\mathbb{R}^2)$, the second term on the r.h.s of the above inequality is estimated as in \eqref{eqmainT3}. Now, let $m_2$ be an integer fixed such that $2m_2>\bar{s}+\frac{\alpha+1}{2}$, then by Proposition \ref{CalderonComGU} we find
\begin{equation*}
\begin{aligned}
\|[\vartheta_{\epsilon,b},&\mathcal{H}_x]D_x^{\frac{\alpha+1}{2}}J^{\bar{s}}_x(u\psi_{\epsilon})\|_{L^2}\\
&=\|\|[\vartheta_{\epsilon,b},\mathcal{H}_x]J^{2m_2}_xJ^{-2m_2}_x D_x^{\frac{\alpha+1}{2}}J^{\bar{s}}_x(u\psi_{\epsilon})\|_{L^2_x}\|_{L^2_y} \\
&\lesssim \|J^{2m_2}_x(\vartheta_{\epsilon,b})\|_{L^{\infty}}\|J^{-2m_2}_x D_x^{\frac{\alpha+1}{2}}J^{\bar{s}}_x(u\psi_{\epsilon})\|_{L^2_{xy}}\\ 
&\lesssim \|u\|_{L^2_{xy}}.
\end{aligned}
\end{equation*}
This completes the estimate for \eqref{eqmainT2} and in turn the study of $A_{2,1}(t)+A_{2,2}(t)+A_{2,3}(t)$.

Finally, we deal with $A_{2,4}(t), A_{2,5}(t)$ and $A_{2,6}(t)$ defined in \eqref{eqpropa0}. By using again that $1=\chi_{\epsilon,b}+\phi_{\epsilon,b}+\psi_{\epsilon}$, we write
\begin{equation}\label{eqmainT3.1}
\begin{aligned}
A_{2,4}(t)+A_{2,5}(t)+A_{2,6}(t)=&2\sigma(J^{\bar{s}}_x u,J^{\bar{s}}_x(u\psi_{\epsilon}))\\
&-\sigma(J^{\bar{s}}_x(u\psi_{\epsilon}),J^{\bar{s}}_x(u\psi_{\epsilon})).
\end{aligned}
\end{equation}
We proceed to estimate each factor of the above identity.  By opening up the commutator defining $\sigma$, using that $J^{\bar{s}}$ is a symmetry operator and that $D^{\alpha+1}_x=\mathcal{H}_x D^{\alpha}_x\partial_x$, it is deduced that
\begin{equation*}
\begin{aligned}
2\sigma(J^{\bar{s}}_x u,J^{\bar{s}}_x(u\psi_{\epsilon})) %=&\int J^{\bar{s}}_xu\mathcal{H}_x D_x^{\alpha}\partial_x^2\big(\chi^2_{\epsilon,b}J^{\bar{s}}_x(u\psi_{\epsilon})\big)\, dx dy\\
%&-\int J^{\bar{s}}_x u \chi^2_{\epsilon,b} \mathcal{H}_x D_x^{\alpha}\partial_x^2J^{\bar{s}}_x(u\psi_{\epsilon})\, dx dy\\
=&\int u J^{\bar{s}}_x \mathcal{H}_x D_x^{\alpha}\partial_x^2\big(\chi^2_{\epsilon,b}J^{\bar{s}}_x (u\psi_{\epsilon})\big)\, dx dy\\
&-\int u J^{\bar{s}}_x \big(\chi^2_{\epsilon,b} \mathcal{H}_x D_x^{\alpha}\partial_x^2J^{\bar{s}}_x(u\psi_{\epsilon})\big)\, dx dy \\
=&\int u J^{\bar{s}}_x \mathcal{H}_x D_x^{\alpha}\partial_x^2\big(\chi^2_{\epsilon,b}J^{\bar{s}}_x (u\psi_{\epsilon})\big)\, dx dy\\
&-\int u J^{\bar{s}}_x \big([\chi^2_{\epsilon,b}, \mathcal{H}_x ]D_x^{\alpha}\partial_x^2J^{\bar{s}}_x(u\psi_{\epsilon})\big)\, dx dy \\
&-\int u J^{\bar{s}}_x \mathcal{H}_x \big(\chi^2_{\epsilon,b}D_x^{\alpha}\partial_x^2J^{\bar{s}}_x(u\psi_{\epsilon})\big)\, dx dy \\
=:&\widetilde{A}_{2,3}(t)+\widetilde{A}_{2,4}(t)+\widetilde{A}_{2,5}(t).
\end{aligned}
\end{equation*}
Let $m_3$ be an even integer sufficiently large such that $m_3>\bar{s}+2+\alpha$, then by the embedding $H^{m_3}(\mathbb{R}) \hookrightarrow H^{\bar{s}+2+\alpha}(\mathbb{R})$ on the $x$-variable, we find
\begin{equation}\label{eqmainT4}
\begin{aligned}
|\widetilde{A}_{2,3}(t)|&+|\widetilde{A}_{2,5}(t)|\\
\lesssim &  \sum_{j=0}^{m_3}\big(\|\partial_x^{j}\big(\chi^2_{\epsilon,b}J^{\bar{s}}_x(u\psi_{\epsilon})\big)\|_{L^2_{xy}}\\
&\hspace{2cm}+\|\partial_x^{j}(\chi^2_{\epsilon,b}D^{\alpha}_x\partial_x^2 J^{\bar{s}}_x(u\psi_{\epsilon}))\|_{L^2_{xy}} \big)\Big)\|u\|_{L^2_{xy}}.
\end{aligned}
\end{equation}
Given that $\partial^{k}_xJ^{\bar{s}}_x$ determines a pseudo-differential operator on the $x$-variable for all integer $k\geq 0$, and that $\chi_{\epsilon,b}$ and $\psi_{\epsilon}$ have separated support, by applying Corollary \ref{supportsepareted3} we get
\begin{equation*}
\begin{aligned}
\|\|(\partial_x^k\chi_{\epsilon,b}^2)&(\partial_x^{j-k}J_x^{\bar{s}})(u\psi_{\epsilon})\|_{L^2_x}\|_{L^2_y}+\|\|(\partial_x^k\chi_{\epsilon,b}^2)D_x^{\alpha}(\partial_x^{j-k+2}J_x^{\bar{s}})(u\psi_{\epsilon})\|_{L^2_x}\|_{L^2_y}\\
&\lesssim \|\|\partial_x^k\chi_{\epsilon,b}^2\|_{L^{\infty}}\|u\psi_\epsilon\|_{L^2_x}\|_{L^2_y}\\
&\lesssim \|u\|_{L^2_{xy}},
\end{aligned}
\end{equation*}
for each $k=0,\dots, j$ and $j=0,\dots,m_3$. Summing over these indexes, we control the right-hand side of \eqref{eqmainT4}.

Now, from our choice of $m_3$,
\begin{equation*}
\begin{aligned}
|\widetilde{A}_{2,4}(t)|&\lesssim \|\|J^{m_3}_x[\mathcal{H}_x,\chi_{\epsilon,b}^2]J^{m_3}_x( J^{-m_3}_xD_x^{\alpha}\partial_x^2J^{\bar{s}}_x(u\psi_{\epsilon}))\|_{L^2_x}\|_{L^2_y}\|u\|_{L^2_{xy}}\\
&\lesssim \|J^{2m_3}_x \chi_{\epsilon,b}^2\|_{L^{\infty}}\|\|J^{-m_3}D_x^{\alpha}\partial_x^2J^{\bar{s}}_x(u\psi_{\epsilon})\|_{L^2_x}\|_{L^2_y}\|u\|_{L^2_{xy}}\\
&\lesssim \|u\|_{L^2_{xy}}^2.
\end{aligned}
\end{equation*}
Collecting the previous estimates, we complete the analysis of $\sigma(J^{\bar{s}}_x u,J^{\bar{s}}_x(u\psi_{\epsilon}))$. In light of the fact that the above argument clearly applies to $\sigma(J^{\bar{s}}_x f,J^{\bar{s}}_x(u\psi_{\epsilon}))$ as long as $f\in L^{2}(\mathbb{R}^2)$, the remaining estimate in \eqref{eqmainT3.1} can be controlled in a similar fashion. The proof of the lemma is complete.
\end{proof}

\begin{remark}\label{remacommchi}
It is worth emphasizing that the above considerations dealing with $\widetilde{A}_{2,2}$ in the proof of Lemma \ref{estimatA2} provide the following inequality
\begin{equation*}
\begin{aligned}
\|J^{\bar{s}}(u \chi_{\epsilon,b})\|_{L^2_{xy}} \lesssim \sum_{0\leq j \leq \bar{s}-s_{\alpha}} \|\chi_{\epsilon,b}J^{\bar{s}-j}_xu\|_{L^2_{xy}}+\|J^{\bar{s}-j}_x(u \phi_{\epsilon,b})\|_{L^2_{xy}}+\|u\|_{H^{s_{\alpha}^{+}}},
\end{aligned}
\end{equation*}
whenever $\bar{s}>s_{\alpha}$ is fixed. This estimate will be convenient to replace the analysis of $J^{\bar{s}}(u \chi_{\epsilon,b})$ by that of $\chi_{\epsilon,b}J^{\bar{s}} u$.
\end{remark}
Next, we proceed to deduce Theorem \ref{mainTheo}. We divide our attention into several cases determined by the values of $s>s_{\alpha}=(17-2\alpha)/12$.

%%%%%%%%%%%%%%%%%%%%%%%%%%%%%%%%%%%%%%%%%%%%%%%%%%%%%%%%%%%%%%%%%%%%%%%%%%%%%%%%%%%%%%%%%%%%%%%%%%%%%%%%%%%%%%%%%%%%%%%%%%%%%%%%%%%%%%%%%%%%%%%%%%%%%%%%%%%%%%%%%%%%%%%%%%%%%%%%%%%%%%%%%%%%%%%%%%%%%%%%%%%%%%%%%%%%%%

\subsection{Case: \texorpdfstring{$ s\in (s_{\alpha}, 2)$}{}}

We further divide our consideration into two main steps.
\begin{itemize}
\item {\sc Step 1}. Study differential equation \eqref{identnonl} for $\bar{s}=r+\frac{1-\alpha}{2}$, where $r=s-1$.
\item {\sc Step 2}. Study differential equation \eqref{identnonl} for $\bar{s}=s$.
\end{itemize}
It should be noted that {\sc Step 1} is motivated to obtain the local smoothing effect corresponding to $J^s_x$ derivatives of $u$, which is required to deal with $A_1(t)$ in \eqref{identnonl} for the desired case $\bar{s}=s$. In {\sc Step 2}, we prove Theorem \ref{mainTheo} for indexes $s\in (s_\alpha,2)$. 

\subsubsection{\sc Step 1.}\label{step1}
        
In this part, we deal with identity \eqref{identnonl} for $\bar{s}=r+\frac{1-\alpha}{2}$, where $r=s-1$. We separate our analysis according to the corresponding factors $A_j(t)$, $j=1,2, 3$.

\textbf{\emph{Estimate for $A_1$}}. Since $r=s-1<s_{\alpha}$, Proposition \ref{KATOSM} implies
\begin{equation*}
D_x^{\frac{1+\alpha}{2}}J^r_xu(x,y,t) \in L^2((-R,R)\times \mathbb{R}\times (0,T)),
\end{equation*}
for all $R>0$. Then, in view of the fact that $J^r_xu \in C([0,T],L^{2}(\mathbb{R}^2))$, we can apply Lemma \ref{fraclemma2} (I) to get
 \begin{equation*}
J^{r+\frac{1+\alpha}{2}}_xu(x,y,t) \in L^2((-R,R)\times \mathbb{R}\times (0,T)),
\end{equation*}
for all $R>0$. The previous conclusion and Lemma \ref{fraclemma2} (III) reveal
 \begin{equation}\label{eqmainT0}
J^{\widetilde{s}}_xu(x,y,t) \in L^2((-R,R)\times \mathbb{R}\times (0,T)),
\end{equation}
for all $R>0$ and all $\widetilde{s}\in [0,r+\frac{1+\alpha}{2}]$. Consequently, we set $R_1>0$ such that $\supp(\chi_{\epsilon,b}(x+vt)\chi_{\epsilon,b}'(x+vt))\subset (-R_1,R_1)$, for all $t\in [0,T]$. Then, by noticing that $r+\frac{1-\alpha}{2}<s-1+\frac{1+\alpha}{2}$ with $s\in (s_{\alpha},2)$, and by \eqref{eqmainT0}, it is seen that
\begin{equation}\label{eqmainT0.01}
\begin{aligned}
\int_0^T |A_{1}(t)| \, dt \lesssim \int_0^T \int_{\mathbb{R}}\int_{-R_1}^{R_1} (J_x^{r+\frac{1-\alpha}{2}} u)^2(x,y,t)   \, dx dy dt<\infty.
\end{aligned}
\end{equation} 
The analysis of $A_1(t)$ is complete.

{\bf \emph{Estimate for $A_2$}}. In virtue of Lemma \ref{estimatA2}, we just need to justify the validity of the r.h.s of \eqref{eqlemma1} under the current restrictions. Indeed, since $r+\frac{1-\alpha}{2}-s_\alpha<1$, we are reduced to control
\begin{equation}\label{eqmainT0.1}
 \|\chi_{\epsilon,b}J^{r+\frac{1-\alpha}{2}}_xu\|_{L^2_{xy}}+\|J^{r+\frac{1-\alpha}{2}}_x(u \phi_{\epsilon,b})\|_{L^2_{xy}}.
\end{equation}
The first term on the right-hand side of the above expression is the quantity to be estimate through \eqref{identnonl} and Gronwall's lemma, while the second one is bounded by Lemma \ref{fraclemma2} (IV) and \eqref{eqmainT0} as follows
\begin{equation}\label{eqmainT0.2}
\begin{aligned}
\|J^{r+\frac{1-\alpha}{2}}_x(u \phi_{\epsilon,b})\|_{L^2_{xy}} 
&\lesssim \|\|\phi_{R_1,x}J^{r+\frac{1-\alpha}{2}}_x u\|_{L^2_x}\|_{L^2_y}+\|\|u\|_{L^2_x}\|_{L^2_y}\\
&= \|\phi_{R_1,x}J^{r+\frac{1-\alpha}{2}}_x u\|_{L^2_{xy}}+\|u\|_{L^2_{xy}},
\end{aligned}
\end{equation} 
where $R_1>0$, is given as in \eqref{eqmainT0.01} and we have set $\phi_{R_1,x}=\phi_{R_1,x}(x) \in C^{\infty}_c(\mathbb{R})$ such that $\phi_{R_1,x}(x)=1$ on $[-R_1,R_1]$, and $\dist\big((1-\phi_{R_1,x})(x), \phi_{\epsilon,b}(x+vt)\big)\geq \epsilon/4$, for all $t\in [0,T]$. This completes the considerations for $A_2$.

%%%%%%%%%%%%%%%%%%%%%%%%%%%%%%%%%%%%%%%%%%%%%%%%%%%%%%%%%%%%%%%%%%%%%%%%%%%%%%%%%%%%%%%%%%%%%%%%%%%%%%%%%%%%%%%%%%%%%%%%%%%%%%%%%%%%%%%%%%%%%%%%%

{\bf \emph{Estimate for $A_3$}}. Recalling \eqref{phidecomp} and \eqref{phidecomp2}, we begin by writing 
\begin{equation*}
\begin{split}
\chi_{\epsilon,b}J_{x}^{r+\frac{1-\alpha}{2}}&(u\partial_{x}u)\\
&=-\frac{1}{2}[J_{x}^{r+\frac{1-\alpha}{2}}, \chi_{\epsilon,b}] \partial_{x}\big((u\chi_{\epsilon,b})^{2}+(u\widetilde{\phi}_{\epsilon,b})^{2}+u^{2}\psi_{\epsilon}\big)\\
&\quad + [J^{r+\frac{1-\alpha}{2}}_{x}, u\chi_{\epsilon,b}]\partial_{x}\big(u\chi_{\epsilon,b}+u\phi_{\epsilon,b}+u\psi_{\epsilon}\big)+u\chi_{\epsilon,b}\partial_{x}J_{x}^{r+\frac{1-\alpha}{2}}u\\
&=: A_{3,1}+A_{3,2}+A_{3,3}+A_{3,4} +A_{3,5}+A_{3,6}+ A_{3,7}.
\end{split}
\end{equation*}
In the first place, we obtain, after applying Lemmas \ref{fraLR} and \ref{lem1},
\begin{equation*}
\begin{split}
\|A_{3,1}\|_{L^{2}_{x,y}}&=\frac{1}{2}\|\|[J_{x}^{r+\frac{1-\alpha}{2}}, \chi_{\epsilon,b}] \partial_{x}(u\chi_{\epsilon,b})^2\|_{L^2_x}\|_{L^2_y}\\
&\lesssim \|\|J_{x}^{l}\chi_{\epsilon,b}'\|_{L^{2}}\|J_{x}^{r+\frac{1-\alpha}{2}}(u\chi_{\epsilon,b})^{2}\|_{L^{2}_{x}}\|_{L^2_y}\\
&\lesssim \|\|u\chi_{\epsilon,b}\|_{L^{\infty}_{x}}\|J_{x}^{r+\frac{1-\alpha}{2}}(u\chi_{\epsilon,b})\|_{L^{2}_{x}}\|_{L^2_y}\\
&\lesssim \|u\|_{L_{xy}^{\infty}}\|J_{x}^{r+\frac{1-\alpha}{2}}\left(u\chi_{\epsilon,b}\right)\|_{L_{xy}^{2}}.
\end{split}
\end{equation*}
A similar analysis can be applied to provide upper bounds for  $A_{3,2}$ as follows 
\begin{equation*}
\|A_{3,2}\|_{L^{2}_{xy}}\lesssim \|u\|_{L_{xy}^{\infty}}\|J_{x}^{r+\frac{1-\alpha}{2}}(u\widetilde{\phi}_{\epsilon,b})\|_{L_{xy}^{2}}.
\end{equation*}
Instead, for $A_{3,3}$, by opening the commutator involved, we take hand of the relationship of the weighted functions involved. More precisely, since
\begin{equation*}
\dist( \supp(\chi_{\epsilon,b}), \supp(\psi_{\epsilon}))\geq \frac{\epsilon}{2}>0,
\end{equation*}
we obtain after applying Lemma \ref{supportseparetedJ} on the $x$-spatial variable the bound 
\begin{equation*}
\begin{split}
\|A_{3,3}\|_{L^{2}_{xy}}&=\|\|\chi_{\epsilon,b} J_{x}^{r+\frac{1-\alpha}{2}}(u^2\psi_{\epsilon})\|_{L^2_x}\|_{L^2_y}\\
&\lesssim \|\|\chi_{\epsilon,b}\|_{L^{\infty}}\|u^2\psi_{\epsilon}\|_{L^2_x}\|_{L^2_y}\\
&\lesssim \|u\|_{L_{xy}^{\infty}}\|u\|_{L^{2}_{xy}}.
\end{split}
\end{equation*}
A similar analysis applied to $A_{3,6}$ produces
\begin{equation*}
\|A_{3,6}\|_{L^{2}_{xy}}\lesssim \|u\|_{L^{\infty}_{xy}}\|u\|_{L^{2}_{xy}}.
\end{equation*} 
Instead, the terms $A_{3,4}$ and $A_{3,5}$ require implementing more sophisticated tool. In this sense, Kato-Ponce commutator estimate Lemma \ref{conmKP} guarantee that
\begin{equation*}
\begin{split}
\|A_{3,4}\|_{L^{2}_{xy}}&=\|\|[J^{r+\frac{1-\alpha}{2}}_{x},u\chi_{\epsilon,b}]\partial_{x}(u\chi_{\epsilon,b})\|_{L^2_x}\|_{L^2_{y}}\\
&\lesssim \|\|\partial_x(u\chi_{\epsilon,b})\|_{L^{\infty}_x}\|J^{r+\frac{1-\alpha}{2}}_{x}(u\chi_{\epsilon,b})\|_{L^2_x}\|_{L^2_y} \\
&\lesssim (\|u\|_{L^{\infty}_{xy}}+\|\partial_{x}u\|_{L^{\infty}_{xy}})\|J_{x}^{r+\frac{1-\alpha}{2}}(u\chi_{\epsilon,b})\|_{L_{xy}^{2}}.
\end{split}
\end{equation*}
For $A_{3,5}$, the same combination of techniques as before allow us to obtain 
\begin{equation*}
\begin{split}
\|A_{3,5}\|_{L^{2}_{xy}}& \lesssim (\|u\|_{L^{\infty}_{xy}}+\|\partial_{x}u\|_{L^{\infty}_{xy}})\big(\|J_{x}^{r+\frac{1-\alpha}{2}}(u\chi_{\epsilon,b})\|_{L^2_{xy}}+\|J_{x}^{r+\frac{1-\alpha}{2}}(u\phi_{\epsilon,b})\|_{L^2_{xy}}\big).
\end{split}
\end{equation*}
Finally, the term $A_{3,7}$ can be handled by going back at the integral defining $A_{3}$ and integrating by parts. More precisely,
\begin{equation*}
\begin{split}
\int A_{3,7}&\chi_{\epsilon,b}J_{x}^{r+\frac{1-\alpha}{2}}u\,dx\,dy \\
&=\int u\chi_{\epsilon,b}\partial_{x}J_{x}^{r+\frac{1-\alpha}{2}}u (\chi_{\epsilon,b}J_{x}^{r+\frac{1-\alpha}{2}}u)\,dx\,dy\\
&=-\frac{1}{2}\int \partial_{x}u\chi_{\epsilon,b}^{2}\big(J_{x}^{r+\frac{1-\alpha}{2}}u\big)^{2}dx\,dy-\int u\chi_{\epsilon,b}\chi_{\epsilon,b}'\big(J_{x}^{r+\frac{1-\alpha}{2}}u\big)^2\, dx\, dy\\
&=A_{3,7,1}(t)+A_{3,7,2}(t).
\end{split}
\end{equation*}
Notice that the first term on the r.h.s. can be bounded by  using the Strichartz estimate provided by the local theory, that is 
\begin{equation*}
\begin{split}
|A_{3,7,1}(t)|&\lesssim \left\|\partial_{x}u\right\|_{L^{\infty}_{xy}}\| \chi_{\epsilon,b}J_{x}^{r+\frac{1-\alpha}{2}}u\|_{L^2_{xy}}^2,
\end{split}
\end{equation*}
being the last expression the one to be estimated after using Gronwall's inequality. 

We emphasize that in spite of $A_{3,7,1}$ and $A_{3,7,2}$ share some similarities, the regions involved in each of these estimates are different, and so the way we can provide some upper bounds for their respective expressions. In this sense, we obtain by Sobolev embedding
\begin{equation*}
\begin{split}
|A_{3,7,2}(t)|&\leq \|u\|_{L^{\infty}_{xy}}\|\chi_{\epsilon,b}\chi_{\epsilon,b}'J_{x}^{r+\frac{1-\alpha}{2}}u\|_{L^2_{xy}}^{2} \lesssim \|u\|_{H^{s_{\alpha}^{+}}}A_1(t),
\end{split}
\end{equation*}
where the term  $A_{1}$ was already bounded at the beginning of this section.

In order to fully control the above estimates generated in the study of the non-linear part, it remains to control the terms 
\begin{equation}\label{eqmainT0.3} 
\|J_{x}^{r+\frac{1-\alpha}{2}}(u\chi_{\epsilon,b})\|_{L^{2}_{xy}}, \, \|J_{x}^{r+\frac{1-\alpha}{2}}(u\phi_{\epsilon,b})\|_{L^{2}_{xy}}, \text{ and } \|J_{x}^{r+\frac{1-\alpha}{2}}(u\widetilde{\phi}_{\epsilon,b})\|_{L^{2}_{xy}}.
\end{equation}
Recalling that the $\supp(\phi_{\epsilon,b}),\supp(\widetilde{\phi}_{\epsilon,b})\subset [\epsilon/4,b]$, the same reasoning around \eqref{eqmainT0.2} allow us to bound $\|J_{x}^{r+\frac{1-\alpha}{2}}(u\phi_{\epsilon,b})\|_{L^{2}_{xy}}$ and $\|J_{x}^{r+\frac{1-\alpha}{2}}(u\widetilde{\phi}_{\epsilon,b})\|_{L^{2}_{xy}}$. 

On the other hand, given that $r+\frac{1-\alpha}{2}-s_{\alpha}=s-1+\frac{1-\alpha}{2}-s_{\alpha}<1$, $s\in(s_{\alpha},2)$, by Remark \ref{remacommchi} it is seen that
\begin{equation*}
\begin{aligned}
\|J^{r+\frac{1-\alpha}{2}}(u \chi_{\epsilon,b})\|_{L^2_{xy}} \lesssim \|\chi_{\epsilon,b}J^{r+\frac{1-\alpha}{2}}_xu\|_{L^2_{xy}}+\|J^{r+\frac{1-\alpha}{2}}_x(u \phi_{\epsilon,b})\|_{L^2_{xy}}+\|u\|_{H^{s_{\alpha}^{+}}},
\end{aligned}
\end{equation*}
where the first term on the right-hand side of the above expression is the quantity to be estimated employing Gronwall's lemma, and the second term was considered in previous discussions. Thereby, the estimate for the factor arising from the nonlinear $A_3(t)$ is complete. 

\begin{remark}\label{remarA3}
The arguments in the above estimate are quite general and they can be implemented to  provide a bound for the factor $A_3(t)$ in \eqref{identnonl} for any regularity $\bar{s}>s_{\alpha}$. Indeed, we just need to justify the terms \eqref{eqmainT0.3} obtained after replacing the regularity $r+\frac{1-\alpha}{2}$ by $\bar{s}$, that is to say,
\begin{equation}\label{eqremarA3}
\begin{aligned}
\sum_{0\leq j \leq \bar{s}-s_{\alpha}} \|\chi_{\epsilon,b}J^{\bar{s}-j}_xu\|_{L^2_{xy}}+&\|J^{\bar{s}-j}_x(u \phi_{\epsilon,b})\|_{L^2_{xy}}\\
&+\|J^{\bar{s}}_x(u \widetilde{\phi}_{\epsilon,b})\|_{L^2_{xy}}+\|u\|_{H^{s_{\alpha}^{+}}},
\end{aligned}
\end{equation}
where we have applied Remark \ref{remacommchi} to bound $\|J^{\bar{s}}(u \chi_{\epsilon,b})\|_{L^2}$.
\end{remark}

Finally, recalling $Sm_1(t)$ and $Sm_2(t)$ defined in Lemma \ref{estimatA2}, we gather the estimates above to conclude that there exist some constants $c_0$ and $c_1$ such that 
\begin{equation*}
\begin{aligned}
\frac{1}{2}\frac{d}{dt}\|\chi_{\epsilon,b} J_x^{r+\frac{1-\alpha}{2}} u \|_{L^2_{xy}}^2 & +Sm_1(t)+Sm_2(t)+\|(\chi\chi_{\epsilon,b}')^{1/2}\partial_y  J_x^{r+\frac{1-\alpha}{2}} u\|_{L^2_{xy}}^2 \\
\leq & \frac{1}{4}Sm_1(t)+c_0\|u\|_{L^{\infty}_TH^{s_{\alpha}^{+}}}^2+(1+\|u\|_{L^{\infty}_TH^{s_{\alpha}^{+}}})\mathfrak{g}(t)\\
&+c_1\big(1+\|u\|_{L^{\infty}}+\|\partial_x u\|_{L^{\infty}}\big)\| \chi_{\epsilon,b} J_x^{r+\frac{1-\alpha}{2}} u\|_{L^2_{xy}}^2,
\end{aligned}
\end{equation*}
where the function $\mathfrak{g}(t)$ satisfies
\begin{equation*}
\int_0^T\mathfrak{g}(t) dt \lesssim \int_0^T\int_{\mathbb{R}}  \int_{-R_1}^{R_1}(J_x^{r+\frac{1-\alpha}{2}} u)^2(x,y,t) \, dx dy dt <\infty, 
\end{equation*}
being $R_1>0$ provided by \eqref{eqmainT0.01}. Thus, Gronwall’s inequality and integration in time yield
\begin{equation}\label{concluGronwal}
\begin{aligned}
\sup_{t\in[0,T]}\|\chi_{\epsilon,b}& J_x^{r+\frac{1-\alpha}{2}} u(t)\|_{L^2_{xy}}^2  +\int_0^T Sm_1(t)\, dt\\
&+\int_0^T \int_{\mathbb{R}^2} \chi_{\epsilon,b}\chi_{\epsilon,b}'(\partial_yJ_x^{r+\frac{1-\alpha}{2}}u)^2(x,y,t)\, dx dy dt \lesssim  \widetilde{c},
\end{aligned}
\end{equation}
where $\widetilde{c}=\widetilde{c}(\epsilon,v,T,\|\chi_{\epsilon,b} J_x^{r+\frac{1-\alpha}{2}} u_0\|_{L^2_{xy}},\|u\|_{L^{\infty}_TH^{s_{\alpha}^{+}}},\|\partial_x u\|_{L^1_TL^{\infty}_{xy}})$ and we have used that $\int_0^T Sm_s(t)\, dt \geq 0$. This completes the deduction of {\sc Step 1}.

%%%%%%%%%%%%%%%%%%%%%%%%%%%%%%%%%%%%%%%%%%%%%%%%%%%%%%%%%%%%%%%%%%%%%%%%%%%%%%%%%%%%%%%%%%%%%%%%%%%%%%%%%%%%%%%%%%%%%%%%%%%%%%%%%%%%%%%%%%%%%%%%%%%%%%%%%%%%%%%%%%%%%

\subsubsection{\sc  Step 2.}\label{Step2}

Fixing $\epsilon$ and $b\geq 5 \epsilon$, this part concerns the analysis of \eqref{identnonl} for $\widetilde{s}=s$, $s\in (s_\alpha,2)$, i.e., we establish the proof of Theorem \ref{mainTheo} for such indexes $s$. We shall only estimate the terms $A_j(t)$, $j=1,2,3$. Once this has been done, following similar consideration leading to  \eqref{concluGronwal}, we obtain the desired conclusion. To avoid repetition, we omit these details.

{\bf Estimate for $A_1$}. By applying the arguments in {\sc Step 1} above with the function $\chi^{2}_{\epsilon/24,b+13\epsilon/12}$, and by properties (iii) and (iv) in Subsection \ref{subweigh}, the local smoothing effect obtained in {\sc Step 1} shows
\begin{equation}\label{eqmainT5}
\int_0^T \int_{\mathbb{R}^2} \mathbbm{1}_{[\epsilon/8,b+\epsilon]}(x+vt)   (D^{\frac{1+\alpha}{2}}_x J^{s-1+\frac{1-\alpha}{2}}_x u)^2(x,y,t)\, dxdydt \leq c.
\end{equation}
Additionally, \eqref{eqmainT0} determines
\begin{equation}\label{eqmainT6}
\int_0^T \int_{\mathbb{R}^2} \mathbbm{1}_{[\epsilon/8,b+\epsilon]}(x+vt) ( J^{s-1+\frac{1-\alpha}{2}}_x u)^2(x,y,t)\, dxdydt \leq c.
\end{equation}
Then, we consider functions $\theta_{1,\epsilon,b},\theta_{2,\epsilon,b},\theta_{3,\epsilon,b} \in C^{\infty}_c(\mathbb{R})$ with $0\leq \theta _{j,\epsilon,b} \leq 1$, $j=1,2,3$ such that
\begin{equation}\label{decompfunct}
\begin{aligned}
&\theta_{1,\epsilon,b} \equiv 1 \text{ on } [5\epsilon/32, b+3\epsilon/4],  &&\text{and }\, \supp(\theta_{1,\epsilon,b}) \subset [\epsilon/8,b+\epsilon], \\
&\theta_{2,\epsilon,b} \equiv 1 \text{ on } [7\epsilon/32, b+\epsilon/2],  &&\text{and }\, \supp(\theta_{2,\epsilon,b}) \subset [3\epsilon/16,b+5\epsilon/8], \\
&\theta_{3,\epsilon,b} \equiv 1 \text{ on } [\epsilon, b],  &&\text{and }\, \supp(\theta_{3,\epsilon,b}) \subset [\epsilon/4,b+\epsilon/4].
\end{aligned}
\end{equation}
We shall assume that the above functions act only on the $x$-variable in the following manner, $\theta_{j,\epsilon,b}=\theta_{j,\epsilon,b}(x+vt)$, $j=1,2,3$. Then, from \eqref{eqmainT5} and \eqref{eqmainT6}, we infer that $\theta_{1,\epsilon,b}  J^{s-1+\frac{1-\alpha}{2}}_x u$ and $\theta_{1,\epsilon,b} D^{\frac{1+\alpha}{2}}_x J^{s-1+\frac{1-\alpha}{2}}_x u\in L^{2}(\mathbb{R}^2\times (0,T))$. Additionally, we have $J^{s-1+\frac{1-\alpha}{2}}_x u(\cdot,y,t) \in H^{-m}(\mathbb{R})$ for $m>s-1+\frac{1-\alpha}{2}$ for almost every $y$ and $t$. Therefore, we are in condition to apply Lemma \ref{fraclemma2} (I) to obtain
\begin{equation}\label{eqmainT7}
\begin{aligned}
\|\theta_{2,\epsilon,b} & J^{s}_x u\|_{L^2(\mathbb{R}^2\times (0,T))}\\
=& \|\|\theta_{2,\epsilon,b} J^{s}_xu\|_{L^2_x}\|_{L^2_{yt}(\mathbb{R}\times (0,T))} \\
\lesssim &\|\|\theta_{1,\epsilon,b} J^{s-1+\frac{1-\alpha}{2}}_x u\|_{L^2_x}\|_{L^2_{yt}(\mathbb{R}\times (0,T))}\\
&+\|\|\theta_{1,\epsilon,b} D^{\frac{\alpha+1}{2}}J^{s-1+\frac{1-\alpha}{2}}_x u\|_{L^2_x}\|_{L^2_{yt}(\mathbb{R}\times (0,T))}\\
&+\|\|J^{-m}_x J^{s-1+\frac{1-\alpha}{2}}_x u\|_{L^2_x}\|_{L^2_{yt}(\mathbb{R}\times (0,T))} \\
\lesssim &\|\theta_{1,\epsilon,b} J^{s-1+\frac{1-\alpha}{2}}_x u\|_{L^2_{xyt}(\mathbb{R}^2\times (0,T))}\\
&+\|\theta_{1,\epsilon,b} D^{\frac{\alpha+1}{2}}J^{s-1+\frac{1-\alpha}{2}}_x u\|_{L^2_{xyt}(\mathbb{R}^2\times (0,T))}+\|u\|_{L^2_{xyt}(\mathbb{R}\times (0,T))}.
\end{aligned}
\end{equation}
We conclude that $\theta_{2,\epsilon,b}  J^{s}_x u \in L^2(\mathbb{R}^2 \times (0,T))$. From this fact, and by similar reasoning to \eqref{eqmainT7}, employing Lemma \ref{fraclemma2} (III) instead, we deduce 
\begin{equation}\label{eqmainT8}
\theta_{3,\epsilon,b} J^{\widetilde{s}}u \in L^2(\mathbb{R}^2 \times (0,T)), \, \text{ for any } \,\widetilde{s}\in (0,s].
\end{equation} 
Finally, by \eqref{eqmainT8} and support considerations, we arrive at
\begin{equation*}
\begin{aligned}
\big|\int_0^T A_1(t)\, dt \big|&=|v|\int_0^T \int \chi_{\epsilon,b}\chi_{\epsilon,b}' (J^su)^2 \, dx dy dt \\
&\lesssim \int_0^T \int (\theta_{3,\epsilon,b} J^su)^2\, dx dy dt<\infty. 
\end{aligned}
\end{equation*}
The estimate for $A_1$ is complete.

{\bf Estimate for $A_2$ and $A_3$}. To control $A_2$, given that $0<s-s_{\alpha}<1$,  Lemma \ref{estimatA2} and \eqref{eqlemma1} reduce our efforts to estimate $\|J^{s}_x(u \phi_{\epsilon,b})\|_{L^2_{xy}}$. However, this estimate can be obtained as a consequence of the support of the function $\theta_{2,\epsilon,b}$ defined above and Lemma \ref{fraclemma2} (IV) as follows
\begin{equation}\label{eqmainT9}
\begin{aligned}
\|J^{s}_x(u \phi_{\epsilon,b})\|_{L^2_{xy}}&\lesssim \|\|\theta_{2,\epsilon,b} J^{s}_xu\|_{L^2_x}\|_{L^2_{y}}+\|\|u\|_{L^2_x}\|_{L^2_{y}}\\
&=\|\theta_{2,\epsilon,b}J^{s}_xu\|_{L^2_{xy}}+\|u\|_{L^2_{xy}}.
\end{aligned}
\end{equation}
Since we already proved that $\theta_{2,\epsilon,b}  J^{s}_x u \in L^2(\mathbb{R}^2 \times (0,T))$, the estimate for $A_2$ is complete.

Finally, to estimate $A_3$, according to Remark \ref{remarA3}, we just need to justify \eqref{eqremarA3} for $\bar{s}=s$. However, due to the restriction $s-s_{\alpha}<1$, this is equivalent to study the norms $\|J^{s}_x(u \phi_{\epsilon,b})\|_{L^2_{xy}}$ and $\|J^{s}_x(u \widetilde{\phi}_{\epsilon,b})\|_{L^2_{xy}}$. Recalling that $\supp(\phi_{\epsilon,b}),\supp(\widetilde{\phi}_{\epsilon,b})\subset [\epsilon/4,b]$, these estimates follow from \eqref{eqmainT9}. 

%%%%%%%%%%%%%%%%%%%%%%%%%%%%%%%%%%%%%%%%%%%%%%%%%%%%%%%%%%%%%%%%%%%%%%%%%%%%%%%%%%%%%%%%%%%%%%%%%%%%%%%%%%%%%%%%%%%%%%%%%%%%%%%%%%%%%%%%%%%%%%%%%%%%%%%%%%%%%%%%%%%%%%%%%%%%%%%%%%%%%%%%%%%%%%%%%%%%%%%%%%%%%%%%%%%%%%%%%%%%%%%%%%%%%%%%%%%%%%%%%%%%%%%%%
%%%%%%%%%%%%%%%%%%%%%%%%%%%%%%%%%%%%%%%%%%%%%%%%%%%%%%%%%%%%%%%%%%%%%%%%%%%%%%%%%%%%%%%%%%%%%%%%%%%%%%%%%%%%%%%%%%%%%%%%%%%%%%%%%%%%%%%%%%%%%%%%%%%%%%%%%%%%%%%%%%%%%%

\subsection{Case \texorpdfstring{$k-(k-2)\big(\frac{1-\alpha}{2}\big)\leq s < k+1-(k-1)\big(\frac{1-\alpha}{2}\big)$, $k\geq 2$}{}}

We have developed all the set up required to provide the proof of Theorem \ref{mainTheo} for arbitrary regularity $s>s_{\alpha}=(17-2\alpha)/12$. Indeed, we will proceed employing an inductive argument.

Let $\epsilon'>0$, $b'\geq 5 \epsilon'$ and $k\geq 2$ be given, we shall assume by the inductive hypothesis
\begin{equation}\label{eqindhyp}
\begin{aligned}
\sup_{t\in[0,T]}\|\chi_{\epsilon',b'}J^{\widetilde{s}}_xu(t)&\|_{L^2_{xy}}+\|(\chi_{\epsilon',b'}\chi_{\epsilon',b'}')^{1/2} J^{\widetilde{s}}_x u\|_{L^2(\mathbb{R}^2\times (0,T))}\\
&+
\|(\chi_{\epsilon',b'}\chi_{\epsilon',b'}')^{1/2} D^{\frac{1+\alpha}{2}}_x J^{\widetilde{s}}_xu\|_{L^2(\mathbb{R}^2\times (0,T))} \\
&+\|(\chi_{\epsilon',b'}\chi_{\epsilon',b'}')^{1/2} \partial_y J^{\widetilde{s}}_xu \|_{L^2(\mathbb{R}^2\times (0,T))} \\
&\leq \widetilde{c}_{\widetilde{s}}(\epsilon',b',T,\|\chi_{\epsilon',b'}J^{\widetilde{s}}_x u_0\|_{L^2}, \|u\|_{L^{\infty}_T H^{s_{\alpha}^{+}}},\|\partial_x u\|_{L^1_TL^{\infty}_{xy}}),
\end{aligned}
\end{equation}
whenever $l-(l-2)\big(\frac{1-\alpha}{2}\big)\leq \widetilde{s} < l+1-(l-1)\big(\frac{1-\alpha}{2}\big)$ if $2\leq l<k$, (which holds when $k>2$), or $\widetilde{s}<2$ when $k=2$. It is worth pointing out that the second term on the l.h.s of \eqref{eqindhyp} corresponds to the estimate for $A_1$ in \eqref{identnonl} for $\bar{s}=\widetilde{s}$, while the third and fourth
terms are the smoothing effect granted by the dispersion $-D^{\alpha+1}\partial_x+\partial_x\partial_y^2$ in the equation in \eqref{BOZK}. In addition, we remark that by Lemma \ref{fraclemma2} (III), and the hypothesis $\mathbbm{1}_{\{x>0\}}J^{s}_x u_0 \in L^{2}(\mathbb{R}^2)$, we have
\begin{equation*}
\begin{aligned}
\|\chi_{\epsilon',b'}J^{\widetilde{s}}_x u_0\|_{L^2_{xy}} \lesssim \|\mathbbm{1}_{\{x>0\}}J^{s}_x u_0\|_{L^2_{xy}}+\|u_0\|_{L^2_{xy}},
\end{aligned}
\end{equation*}
whenever $\widetilde{s}\in [0,s]$, which justify the validity of the implicit energy estimates behind the inductive hypothesis.

Setting in \eqref{identnonl}, $\epsilon>0$, $b\geq 5\epsilon$ and $\bar{s}=s$ with $k-(k-2)\big(\frac{1-\alpha}{2}\big)\leq s < k+1-(k-1)\big(\frac{1-\alpha}{2}\big)$, $k\geq 2$, the desired estimate is obtained after controlling the respective factors $A_1,A_2$ and $A_3$.

{\bf Estimate for $A_1$.} We consider the functions $\theta_{1,\epsilon,b},\theta_{2,\epsilon,b}$ and $\theta_{3,\epsilon,b}$ determined by \eqref{decompfunct} with $\theta_{j,\epsilon,b}=\theta_{j,\epsilon,b}(x+vt)$, $j=1,2,3$. By assumption \eqref{eqindhyp} with $\widetilde{s}=s-1+\frac{1-\alpha}{2}$, $\epsilon'=\epsilon/24$ and $b'=b+13\epsilon/12$, we infer
\begin{equation*}
\theta_{1,\epsilon,b} J^{s-1+\frac{\alpha-1}{2}}_xu, \, \, \theta_{1,\epsilon,b} D_x^{\frac{\alpha+1}{2}}J^{s-1+\frac{\alpha-1}{2}}_x u \in L^{2}(\mathbb{R}^2\times(0,T)).
\end{equation*}
Then, by Lemma \ref{fraclemma2} (I) and similar considerations to \eqref{eqmainT7},  the above consequence establishes
\begin{equation}\label{eqmainT9.1}
\theta_{2,\epsilon,b} J^{s}_x u \in L^{2}(\mathbb{R}^2\times(0,T)),
\end{equation}
and so from Lemma \ref{fraclemma2} (III),
\begin{equation*}
\theta_{3,\epsilon,b} J^{s_{\ast}}_x u \in L^{2}(\mathbb{R}^2\times(0,T)), \, \text{  whenever }  \, s_{\ast}\in (0,s].
\end{equation*}
Since $\theta_{3,\epsilon,b}(x+vt)=1$ on $[\epsilon-vt,b-vt]$, the above display yields the desired estimate for $A_1(t)$.

{\bf Estimate for $A_2$ and $A_3$.} To estimate $A_2$ and $A_3$, by previous arguments relaying on Lemma \ref{estimatA2} and Remark \ref{remarA3}, it is enough to control the following expression
\begin{equation}\label{eqmainT10}
\begin{aligned}
\sum_{0\leq j \leq s-s_{\alpha}} \|\chi_{\epsilon,b}J^{s-j}_xu\|_{L^2_{xy}}+\|J^{s-j}_x(u \phi_{\epsilon,b})\|_{L^2_{xy}}+\|J^{\bar{s}}_x(u \widetilde{\phi}_{\epsilon,b})\|_{L^2_{xy}}.
\end{aligned}
\end{equation}
When $j=0$, $ \|\chi_{\epsilon,b}J^{s-j}_xu\|_{L^2_{xy}}$ is the quantity to be estimated by \eqref{identnonl} and Gronwall's lemma. Whereas $\|\chi_{\epsilon,b}J^{s-j}_xu\|_{L^2_{xy}}$ for $1<j\leq s-s_{\alpha}$ are controlled by the inductive hypothesis \eqref{eqindhyp}. 

On the other hand, by support consideration, there exists $\widetilde{\theta}_{2,\epsilon,b}\in C^{\infty}_c(\mathbb{R})$ with  $0\leq \widetilde{\theta}_{2,\epsilon,b} \leq 1$ such that $\dist(1-\theta_{2,\epsilon,b},\widetilde{\theta}_{2,\epsilon,b})\geq \epsilon/32$ and $\dist(1-\widetilde{\theta}_{2,\epsilon,b},\phi_{\epsilon,b})\geq \epsilon/32$. Thus, we apply Lemma \ref{fraclemma2} (IV) followed by part (III) to obtain
\begin{equation*}
\begin{aligned}
\|J^{s-j}_x(u \phi_{\epsilon,b})\|_{L^2_{xy}} &\lesssim \|\widetilde{\theta}_{2,\epsilon,b} J^{s-j}_x u\|_{L^2_{xy}}+\|u\|_{L^2_{xy}} \\
&\lesssim  \|\theta_{2,\epsilon,b} J^{s}_x u\|_{L^2_{xy}}+\|u\|_{L^2_{xy}},
\end{aligned}
\end{equation*}
which is controlled by \eqref{eqmainT9.1} for all integer $0\leq j \leq s-s_{\alpha}$. Noticing that $\supp(\widetilde{\phi}_{\epsilon,b}),\supp(\phi_{\epsilon,b})\subset[\epsilon/4,b]$, the estimate for $\|J^{s}_x(u \widetilde{\phi}_{\epsilon,b})\|_{L^2_{xy}}$ is obtained by the same argument above. These comments provide a control to \eqref{eqmainT10}, and so the estimates for $A_2$ and $A_3$ are complete.

Finally, gathering the previous results and applying Gronwall's inequality, we deduce \eqref{eqindhyp} for $s\in [k-(k-2)\big(\frac{1-\alpha}{2}\big),k+1-(k-1)\big(\frac{1-\alpha}{2}\big))$, $k\geq 2$. This completes the inductive step and in consequence the proof of Theorem \ref{mainTheo}.

%%%%%%%%%%%%%%%%%%%%%%%%%%%%%%%%%%%%%%%%%%%%%%%%%%%%%%%%%%%%%%%%%%%%%%%%%%%%%%%%%%%%%%%%%%%%%%%%%%%%%%%%%%%%%%%%%%%%%%%%%%%%%%%%%%%%%%%%%%%%%%%%%%%%%%%%%%%%%%%

\section{Appendix: Proof of Theorem \ref{imprwellpos}}

This section is aimed to prove Theorem \ref{imprwellpos}. Since our arguments follow similar considerations employed in \cite{FRAC,OscarHBO,KenigKP}, we will state the main ingredients and differences needed to implement these ideas for the IVP \eqref{BOZK}. 

Let us first introduce some notation to be employed along with our arguments.

Consider $\varrho\in C_{c}^{\infty}(\mathbb{R})$ such that
\begin{equation*}
    0\leq \varrho \leq 1, \hspace{0.2cm} \varrho(\xi)=1 \text{ for } |\xi| \leq 1, \hspace{0.2cm} \varrho(\xi)=0 \text{ for } |\xi|\geq 2,
\end{equation*}
and $\varrho_0(\xi)=\varrho(\xi)-\varrho(2\xi)$ which is supported on $1/2\leq |\xi|\leq 2$. For any $f\in S(\mathbb{R}^2)$ and $j\in \mathbb{Z}$, we define the Littlewood-Paley projection operators
\begin{equation}\label{LPprojector}
    \begin{aligned}
    &\widehat{P_{j}^xf}(\xi,\eta)=\varrho_0(2^{-j}\xi)\widehat{f}(\xi,\eta). \\
%    &\widehat{P_{j}f}(\xi,\eta)=\varrho_0(2^{-j}|(\xi,\eta)|)\widehat{f}(\xi,\eta),
    \end{aligned}
\end{equation}

%%%%%%%%%%%%%%%%%%%%%%%%%%%%%%%%%%%%%%%%%%%%%%%%%%%%%%%%%%%%%%%%%%%%%%%%%%%%%%%%%%%%%%%%%%%%%%%%%%%%%%%%%%%%%%%%%%%%%%%%%%%%%%%%%%%%%%%%%%%%%%%%%%%%%%%%%%%%%%%%%%%%%%%%%%%%%%%%%%%%%%%%%%%%%%%%%%%%%%%%%%%%%%%%%%%%%%%%%%%%%%%

\subsection{Strichartz estimates}

The homogeneous problem associated to \eqref{BOZK} is given by
\begin{equation}\label{LBOZK}
    \left\{\begin{aligned}
    &\partial_t u -D^{\alpha+1}_x u_x+u_{xyy} =0,\quad (x,y,t)\in \R^{3}, \, 0\leq\alpha\leq 1, \\
    &u(x,y,0)=\phi(x,y).
    \end{aligned}\right.
\end{equation}
Sufficiently regular solutions of \eqref{LBOZK} will be denoted as
\begin{equation*}
    S(t)\phi(x,y):=\int_{\mathbb{R}^2} \widehat{\phi}(\xi,\eta)e^{i\xi|\xi|^{\alpha+1}t+i\xi\eta^2t+ix\xi+iy\eta} \, d\xi d\eta, 
\end{equation*}
for all $t\in \mathbb{R}$. The result in \cite[Proposition 3]{RibVento} establishes the following estimate for solutions of \eqref{LBOZK}
\begin{equation}\label{eqappend1}
    \left\|P_j^x S(t)\phi\right\|_{L^{p}_{xy}}\lesssim \frac{2^{-\frac{j}{3}(\alpha+\frac{1}{2})(1-\frac{2}{p})}}{|t|^{\frac{5}{6}(1-\frac{2}{p})}} \left\|\phi\right\|_{L^{p'}_{xy}}, \hspace{0.2cm} \text{for all } j \in \mathbb{Z},
\end{equation}
 whenever $1/p+1/ p' =1$ with $p\geq 2$. Hence, an application of the $TT^{\ast}$-argument yields the estimate.

\begin{lemma}\label{STRE}
Let $0\leq \alpha <1$, $p<\infty$ and $j\in \mathbb{Z}$. Then the following estimate holds
\begin{equation}\label{STREeq1}
    \left\| P^x_jS(t)\phi\right\|_{L^{q}_t L^p_{xy}} \lesssim 2^{-\frac{j}{6}(\alpha+\frac{1}{2})(1-\frac{2}{p})} \left\|\phi\right\|_{L^2_{xy}},
\end{equation}
for all $\frac{2}{q}+\frac{5}{3p}=\frac{5}{6}$.
\end{lemma}

The same argument in the proof of \cite[Proposition 3]{RibVento} establishes the following \emph{rough} dispersive estimate,
 \begin{equation}\label{eqappend1.1}
 \left\|P_j^xS(t)\phi\right\|_{L^{p}_{xy}}\lesssim \frac{2^{\frac{j}{2}(1-\frac{2}{p})}}{|t|^{\frac{1}{2}(1-\frac{2}{p})}} \left\|\phi\right\|_{L^{p'}_{xy}},    
\end{equation}
for all $j\in \mathbb{Z}$, $1/p+1/p'=1$ with $p\geq 2$. This result is convenient for our purposes to avoid the negative derivative carried by the right-hand side of \eqref{STREeq1} near the origin in the frequency domain.  Consequently, \eqref{eqappend1.1} provide the following estimate.
\begin{lemma}\label{STREROUGH}
Let $0\leq \alpha < 1$, $p<\infty$ and $j\in \mathbb{Z}$. Then it holds
\begin{equation}\label{STREeq1.1}
    \left\|P_j^x S(t)\phi\right\|_{L^{q}_t L^p_{xy}} \lesssim 2^{\frac{j}{4}(1-\frac{2}{p})}\left\|\phi\right\|_{L^2_{xy}},
\end{equation}
whenever $\frac{2}{q}+\frac{1}{p}=\frac{1}{2}$.
\end{lemma}
%\begin{lemma}\label{STREROUGH}
%Let $p<\infty$ and $\alpha\in [0,1]$. Then it holds
%\begin{equation}\label{STREeq1.1}
%    \left\| S(t)\phi\right\|_{L^{q}_t L^p_{xy}} \lesssim \left\|D_x^{s}\phi\right\|_{L^2_{xy}},
%\end{equation}
%whenever $\frac{2}{q}+\frac{1}{p}=\frac{1}{2}$ and $s=\frac{1}{4}(1-\frac{1}{p})$.
%\end{lemma}
Notice that the case $(q,p)=(2,\infty)$ is not part of the conclusions in Lemmas \ref{STRE} and \ref{STREROUGH}. Accordingly, we require some additional regularity to control this norm.
\begin{cor}\label{liestCoro}
Let $0\leq \alpha < 1$ and $j\in\mathbb{Z}$. For each $T>0$, $0<\delta<1$ and $\theta \in [0,1]$, there exists $\widetilde{k}_{\delta,\theta}\in(0,1/2)$ such that
\begin{equation}\label{STREeq2}
    \left\|P_j^xS(t)f\right\|_{L^2_T L^{\infty}_{xy}}\leq c_{\delta} T^{\widetilde{k}_{\delta,\theta}}2^{\max\{j(\frac{1}{4}-\frac{\theta}{3}(1+\frac{\alpha}{2})),j(\frac{1}{4}-\frac{\theta}{3}(1+\frac{\alpha}{2}))(1-\delta)\}}\left\|J^{\delta}f\right\|_{L^2_{xy}}.
\end{equation}
\end{cor}

\begin{proof}
Let us consider $p<\infty$ sufficiently large to assure that $\delta> \frac{2}{p}$. Combining Sobolev's embedding and \eqref{STREeq1}, we get
\begin{equation}\label{STREeq1.2}
\begin{aligned}
    \|P_j^x & S(t) f \|_{L^2_T L^{\infty}_{xy}} \\
    &\lesssim_{\delta} T^{\frac{q-2}{2q}} \left\| P_j^xS(t)J^{\delta}f\right\|_{L^q_T L^p_{xy}} \lesssim_{\delta}  T^{\frac{q-2}{2q}} 2^{-\frac{j}{6}(\alpha+\frac{1}{2})(1-\frac{2}{p})} \left\|J^{\delta}f\right\|_{L^2_{xy}},
\end{aligned}
\end{equation}
and using \eqref{STREeq1.1} instead, 
\begin{equation}\label{STREeq1.3}
\begin{aligned}
    \|P_j^x & S(t) f \|_{L^2_T L^{\infty}_{xy}} \\
    &\lesssim_{\delta} T^{\frac{\widetilde{q}-2}{2q}} \left\| P_j^xS(t)J^{\delta}f\right\|_{L^{\widetilde{q}}_T L^p_{xy}} \lesssim_{\delta}  T^{\frac{\widetilde{q}-2}{2\widetilde{q}}}  2^{\frac{j}{4}(1-\frac{2}{p})}\left\|J^{\delta}f\right\|_{L^2_{xy}},
\end{aligned}
\end{equation}
where $\frac{2}{\widetilde{q}}+\frac{5}{3p}=\frac{5}{6}$ and $\frac{2}{\widetilde{q}}+\frac{1}{p}=\frac{1}{2}$. %$(1-2/p) \to 1^{+}$ as $p \to \infty$,  
 Interpolating \eqref{STREeq1.2} and \eqref{STREeq1.3} yields \eqref{STREeq2}.
\end{proof}

The preceding conclusion is essential to derive the following Strichartz estimate, which is proved in much the same way as in \cite{KenigKP}.

\begin{lemma} \label{refinStri}
Assume $0\leq \alpha < 1$, $0<\delta <1$, $T>0$ and $s>s_{\alpha}-1=(17-2\alpha)/12-1$. Then, there exists $k_{\delta}\in (\frac{1}{2},1)$ such that
\begin{equation}\label{STREeq3.1}
\begin{aligned}
    \| w &\|_{L^1_TL^{\infty}_{xy}}\\
    &\lesssim_{\delta} T^{k_{\delta}} \big(\sup_{t\in[0,T]}\left\|J^{s+\delta}_x J^{\delta}w(t)\right\|_{L^2_{xy}}+\int_0^T \left\|J^{s-1+\delta}_x J^{\delta}F(\cdot,\tau) \right\|_{L^2_{xy}} \, d\tau\big),
\end{aligned}
\end{equation}
whenever $w$ is a solution of $\partial_tw-D_{x}^{\alpha+1} w_x+ w_{xyy}=F$.
\end{lemma}

\begin{proof} Recalling the projectors introduced in \eqref{LPprojector}, an application of the triangle inequality reduces our considerations to control the r.h.s of the following expression
\begin{equation}\label{STREeq4}
    \left\|w\right\|_{L^1_T L^{\infty}_{xy}} \lesssim \sum_{j}  \left\| P_j^x w\right\|_{L^1_T L^{\infty}_{xy}}.
\end{equation}
Since $P_{j}^xw$ satisfies the integral equation
\begin{equation*}
    P_{j}^xw(t)=S(t)P_{j}^xw(0)+\int_{0}^t S(t-\tau)P_{j}^xF(\cdot,\tau)\, d\tau,
\end{equation*}
by writing $P_j^x=\widetilde{P}_j^xP_j^x$ for some adapted projection $\widetilde{P}_j^x$, we first apply H\"older's inequality and then Corollary \ref{liestCoro} with $\theta=0$ to get
\begin{equation*}
    \begin{aligned}
    &\left\| P_{j}^xw\right\|_{L^1_T L^{\infty}_{xy}} \\
    & \hspace{0.5cm}\lesssim T^{1/2+\widetilde{k}_{\delta,0}}\, 2^{\frac{j}{4}(1-\delta)}\big( \left\|J^{\delta}P_{j}^xw(0)\right\|_{L^2_{xy}}+\int_0^T \left\|J^{\delta}P_{j}^xF(\tau)\right\|_{L^2_{xy}} \, d\tau \big) \\
    & \hspace{0.5cm} \lesssim T^{1/2+\widetilde{k}_{\delta,0}} \, 2^{\frac{j}{4}(1-\delta)}  \big(\sup_{t\in[0,T]} \left\|J^{\delta}w(t)\right\|_{L^2_{xy}}+\int_0^T \left\|J^{\delta}F(\tau)\right\|_{L^2_{xy}} \, d\tau \big),
    \end{aligned}
\end{equation*}
for each $j\leq 0$. Adding the above expression over $j\leq 0$, we derived the desired estimate for these indexes.

To bound the remaining sum on the right-hand side of \eqref{STREeq4}, let us consider $j>0$ and we split the interval $[0,T]=\bigcup_m I_m$, where $I_m=[a_m,b_m]$ and $(b_m-a_m) =c T/2^j$. As a consequence 
\begin{equation}\label{liest5}
    \left\|P_j^x w \right\|_{L^1_T L^{\infty}_{xy}}  \lesssim \frac{T^{1/2}}{2^{j/2}}  \sum_{m} \left\|P_j^x w\right\|_{L^2_{I_m} L^{\infty}_{xy}}.
\end{equation}
By employing Duhamel's formula on each $I_m$, we obtain 
\begin{equation*}
    P_{j}^xw(t)=S(t-a_m)P_j^xw(\cdot,a_m)+\int_{a_k}^t S(t-\tau)P_j^xF(\cdot,\tau)\, d\tau,
\end{equation*}
whenever $t\in I_m$. Then, \eqref{liest5} and Corollary \ref{liestCoro} with $\theta=1$ show 
\begin{equation}\label{STREeq5}
    \begin{aligned}
    \|& P_j^xw \|_{L^1_T L^{\infty}_x} \\
    & \hspace{0.2cm}\lesssim  T^{1/2+\widetilde{k}_{\delta,1}} 2^{-\frac{j}{2}-\frac{j}{6}(\alpha+\frac{1}{2})(1-\delta)}  \sum_{m} \big( \left\|J^{\delta}P_j^xw(a_m)\right\|_{L^2_{xy}}\\
    &\hspace{6.2cm} + \int_{I_m} \left\|J^{\delta}P_j^xF(\tau)\right\|_{L^2_{xy}}\, d\tau \big) \\
    & \hspace{0.2cm}\lesssim  T^{1/2+\widetilde{k}_{\delta,1}} 2^{-\frac{j\delta}{2}} \big( \sup_{t\in [0,T]} \left\|J_x^{\frac{1}{2}-\frac{1}{6}(\alpha+\frac{1}{2})(1-\delta)+\delta/2}J^{\delta}w(t)\right\|_{L^2_{xy}} \\
    &\hspace{3.2cm}+\int_0^T \left\|J^{-\frac{1}{2}-\frac{1}{6}(\alpha+\frac{1}{2})(1-\delta)+\delta/2}_xJ^{\delta}F(\tau)\right\|_{L^2_{xy}}\, d\tau\big).
    \end{aligned} 
\end{equation}
Summing the above expression over $j>0$, using that $s>(17-2\alpha)/12-1=1/2-1/6(\alpha+1/2)$ and that $\frac{1}{6}(\alpha+\frac{1}{2})\delta <\frac{\delta}{2}$, we complete the proof.
\end{proof}

As a further consequence of Lemma \ref{refinStri}, for $0\leq \alpha < 1$, $0<\delta \leq 1$, $T>0$ and $s>s_{\alpha}$, we find that
\begin{equation}\label{STREeq3}
\begin{aligned}
    \|\partial_x w &\|_{L^1_TL^{\infty}_{xy}} \\
    &\lesssim_{\delta} T^{k_{\delta}} \big(\sup_{t\in[0,T]}\left\|J_x^{s +\delta}J^{\delta}w(t)\right\|_{L^2}+\int_0^T \left\|J^{s-1+\delta}_x J^{\delta} F(\cdot,\tau) \right\|_{L^2_{xy}} \, d\tau\big),
\end{aligned}
\end{equation}
and  
\begin{equation}\label{STREeq3.2}
\begin{aligned}
    \|\partial_y w &\|_{L^1_TL^{\infty}_{xy}} \\
    &\lesssim_{\delta} T^{k_{\delta}} \big(\sup_{t\in[0,T]}\left\|J^{s-1 +\delta}_xJ^{1+\delta}w(t)\right\|_{L^2_{xy}}\\
    &\hspace{4cm}+\int_0^T \left\|J^{s-2+\delta}_xJ^{1+\delta}F(\cdot,\tau) \right\|_{L^2_{xy}} \, d\tau\big),
\end{aligned}
\end{equation}
for some $k_{\delta}\in (\frac{1}{2},1)$ and where $w$ solves the equation $\partial_tw-D_{x}^{\alpha+1} w_x+ w_{xyy}=F$.

%%%%%%%%%%%%%%%%%%%%%%%%%%%%%%%%%%%%%%%%%%%%%%%%%%%%%%%%%%%%%%%%%%%%%%%%%%%%%%%%%%%%%%%%%%%%%%%%%%%%%%%%%%%%%%%%%%%%%%%%%%%%%%%%%%%%%%%%%%%%%%%%%%%%%%%%%%%%%%%%%%%%%%%%%%%%%%%%%%%%%%%%%%%%%%%%%%%%%%%%%%%%%%%%%%%%%%%%%%%%%%%%%%%%%%%%

\subsection{Energy Estimates}

Whenever $s>2$, Theorem \ref{imprwellpos} follows by a parabolic regularization argument on \eqref{BOZK}. Roughly speaking, an additional term $-\mu \Delta u$ is added to the equation, after which the limit $\mu \to 0$ is taken. These results follow the same arguments in \cite{CUNHAPAS,Iorio}, so we omit its proof. 

\begin{lemma}\label{comwellp}
Let $s> 2$ and $0\leq \alpha\leq 1$. Then for any $u_0 \in H^s(\mathbb{R}^2)$, there exists $T=T(\left\|u_0\right\|_{H^s})>0$ and a unique solution $u\in C([0,T]; H^s(\mathbb{R}^d))$ of the IVP \eqref{BOZK}. Additionally, the flow-map $u_0 \mapsto u(t)$ is continuous in the $H^s$-norm.
\end{lemma}

For simplicity, we shall take $s=3$ in the above lemma. Therefore, the proof of Lemma \ref{comwellp} also provides existence of smooth solutions. More specifically, given $u_0\in H^{\infty}(\mathbb{R}^2)=\bigcap_{m\geq 0}H^m(\mathbb{R}^2)$ there exist $T(\left\|u\right\|_{H^3})>0$, and a unique solution $u$ of \eqref{BOZK} in the class $C([0,T];H^{\infty}(\mathbb{R}^2))$. Additionally,  Lemma \ref{comwellp} yields the following conclusion.

{\bf Blow-up criteria}. Let $u_0\in H^{\infty}(\mathbb{R}^2)$ there exist $T^{\ast}>T(\left\|u_0\right\|_{H^3})>0$ and a unique maximal solution $u$ of \eqref{BOZK} in $C([0,T^\ast);H^{\infty}(\mathbb{R}^d))$. Moreover, if the maximal time of existence $T^{\ast}$ is finite
\begin{equation}\label{blowupal}
    \lim_{t \to T^{\ast}} \left\|u(t)\right\|_{H^3}=\infty.
\end{equation}

Next, we deduce some estimates involving smooth solutions of \eqref{BOZK}. By recurrent arguments using Lemma \ref{conmKP} to control the nonlinear term in \eqref{BOZK}, it follows:

\begin{lemma}\label{apriEST}
Let $T>0$ and $u\in C([0,T];H^{\infty}(\mathbb{R}^d))$ solution of the IVP associated to \eqref{BOZK}. Then, there exists a positive constant $c_0$ such that
\begin{equation}\label{BasicEnerE} 
    \left\|u\right\|_{L^{\infty}_T H^s}^2 \leq \left\|u_0\right\|_{H^s}^2+c_0 \left\|\nabla u\right\|_{L^1_T L^{\infty}_x}\left\|u\right\|_{L^{\infty}_T H^s}^2
\end{equation}
for any $s>0$.
\end{lemma}
In addition, we require further \emph{a priori} estimates for the  $L^1([0,T];W^{1,\infty}(\mathbb{R}^2))$-norm of smooth solutions of the IVP \eqref{BOZK}.
\begin{lemma}\label{apriEstS} Consider $0\leq \alpha < 1$. Let $u\in C([0,T];H^{\infty}(\mathbb{R}^d))$ be a solution of the IVP \eqref{BOZK}. Then, for any $s>s_{\alpha}$ there exist $k_{\delta}\in(\frac{1}{2},1)$ and $c_s>0$ such that
\begin{equation*}
f(T):= \left\|u\right\|_{L^{1}_TL^{\infty}_{xy}} +\left\|\nabla u \right\|_{L^1_TL^{\infty}_{xy}}.
\end{equation*}
satisfies
\begin{equation}\label{STREeq6}
    f(T)\leq c_sT^{k_{\delta}}(1+f(T))\left\|u\right\|_{L^{\infty}_T H^s}.
\end{equation}
\end{lemma}

\begin{proof}
Let $\widetilde{s} \in (s_{\alpha},s)$ fixed and $0<\delta <\min\left\{(s-\widetilde{s})/2,1\right\}$. From \eqref{STREeq3} with $F=-u\partial_{x}u$, we get
\begin{equation}\label{liest8}
\begin{aligned}
     \|\partial_x u &\|_{L^1_T L^{\infty}_{xy}}\\ 
     &\lesssim T^{\kappa_{\delta}}\big(\sup_{[0,T]} \|J^{\widetilde{s} +2\delta}u(t)\|_{L^2_{xy}}+\int_0^T \|J^{\widetilde{s}-1+2\delta}(u\partial_{x}u)(\tau)\|_{L^2_{xy}} \, d\tau\big),
\end{aligned}
\end{equation}
for some $k_{\delta}\in (1/2,1)$. Our choice of $\delta$ then shows
\begin{equation}\label{liest3}
\sup_{t\in [0,T]}\|J^{\widetilde{s} +2\delta}u(t)\|_{L^2_{xy}} \leq  \|u\|_{L^{\infty}_T H^s},
\end{equation}
and  Lemma \ref{fraLR} gives
\begin{equation}\label{liest4}
    \begin{aligned}
    \|J^{\widetilde{s}-1+2\delta}(u\partial_{x}u) &\|_{L^2_{xy}} \\
    &\lesssim \| u\|_{L^{\infty}}\|J^{\widetilde{s}-1+2\delta}\partial_{x}u\|_{L^{2}_{xy}}+\|J^{\widetilde{s}-1+2\delta} u\|_{L^2_{xy}}\left\|\partial_{x}u\right\|_{L^{\infty}} \\
    &\lesssim \left(\left\|u\right\|_{L^{\infty}}+\left\|\nabla u\right\|_{L^{\infty}}\right)\left\|u\right\|_{L^{\infty}_T H^s}.
    \end{aligned}
\end{equation}
Plugging the above estimates in \eqref{liest8}, we arrive at
\begin{equation}\label{liest11}
   \left\|\partial_x u \right\|_{L^1_T L^{\infty}_{xy}} \lesssim_s T^{\kappa_{\delta}} (1+f(T))\left\|u\right\|_{L^{\infty}_T H^s_x}.
\end{equation} 
On the other hand, \eqref{STREeq3.2} yields
\begin{equation}\label{liest6}
\begin{aligned}
     \|\partial_y u &\|_{L^1_T L^{\infty}_{xy}}\\
      &\lesssim T^{\kappa_{\delta}}\big(\sup_{[0,T]} \|J^{\widetilde{s} +2\delta}u(t)\|_{L^2_{xy}}+\int_0^T \|J_x^{\widetilde{s}-2+\delta}J^{1+\delta}(u\partial_{x}u)(\tau)\|_{L^2_{xy}} \, d\tau\big) \\
     & \lesssim T^{\kappa_{\delta}}\big(\sup_{[0,T]} \|J^{\widetilde{s} +2\delta}u(t)\|_{L^2_{xy}}+\int_0^T \|J^{\widetilde{s}+2\delta}(u^2)(\tau)\|_{L^2_{xy}} \, d\tau\big).
\end{aligned}
\end{equation}
And so the above inequality allows us to argue as in \eqref{liest3} and \eqref{liest4} to obtain 
\begin{equation*}
   \left\|\partial_y u \right\|_{L^1_T L^{\infty}_{xy}} \lesssim_s T^{\kappa_{\delta}} (1+f(T))\left\|u\right\|_{L^{\infty}_T H^s}.
\end{equation*} 
Now, since $\widetilde{s}-1>0$, one can use \eqref{STREeq3.1} and the arguments in \eqref{liest6} to obtain the desired estimate for  $\left\|u\right\|_{L^1_TL^{\infty}_{xy}}$.
\end{proof}

\subsection{A priori Estimates}

We require some additional \emph{a priori} estimates.
\begin{lemma}\label{apriESTlower}
Let $0\leq \alpha < 1$, $s\in (s_{\alpha},3]$. Then there exists $A_s>0$, such that for all $u_0\in H^{\infty}(\mathbb{R}^2)$, there is a solution $u\in C([0,T^{\ast});H^{\infty}(\mathbb{R}^2))$ of \eqref{BOZK}  where $T^{\ast}=T^{\ast}(\left\|u_0\right\|_{H^{3}})>(1+A_s \left\|u_0\right\|_{H^s})^{-2}$. Moreover, there exist a constant $K_0>0$ such that
\begin{equation*}
    \left\|u\right\|_{L^{\infty}_T H^s} \leq 2 \left\|u_0\right\|_{H^s}, 
\end{equation*}
and 
\begin{equation*}
    f(T)=\left\|u\right\|_{L^1_T L^{\infty}_{xy}}+\left\|\nabla u\right\|_{L^1_T L^{\infty}_{xy}} \leq K_0,
\end{equation*}
whenever $T\leq(1+A_s\left\|u_0\right\|_{H^s})^{-2}$.
\end{lemma}
\begin{proof}
In view of the Lemmas \ref{apriEST}, \ref{apriEstS} and the blow-up criteria \eqref{blowupal} applied to the $H^{3}$-norm, the proof follows from the same arguments in \cite[Lemma 5.3]{FKP}.
\end{proof}

\subsection{Proof of Theorem \ref{imprwellpos}}

According to Lemma \ref{comwellp}, we shall assume that $(17-2\alpha)/12=s_{\alpha}<s\leq 3$. Let us consider  $u_0\in H^{s}(\mathbb{R}^2)$ fixed. The existence part is deduced employing the Bona-Smith argument \cite{BS}. More specifically,  we regularize the initial data by choosing $\rho\in C_c^{\infty}(\mathbb{R}^{2})$ radial with $0\leq \rho\leq 1$, $\rho(\xi,\eta)=1$ for $|(\xi,\eta)|\leq 1$ and $\rho(\xi,\eta)=0$ for $|(\xi,\eta)|>2$, we set then
$$u_{0,n}:=\left\{\rho(\xi/n,\eta/n)\widehat{u}_0(\xi,\eta)\right\}^{\vee}$$
for any integer $n\geq 1$. Now, by employing Plancherel's identity and Lebesgue dominated convergence theorem, it is not difficult to see that for $m\geq n\geq 1$,
\begin{equation}\label{estimapprox}
n^{\sigma}\|J^{s-\sigma}(u_{0,n}-u_{0,m})\|_{L^2_{xy}}  \underset{n\to \infty}{\rightarrow} 0,
\end{equation}
whenever $0\leq \sigma \leq s$.

Consequently, since $\left\|u_{0,n}\right\|_{H^{s}}\leq \|u_{0}\|_{H^{s}}$, Lemma \ref{apriESTlower} establishes the existence of regularized solutions $u_n \in C([0,T];H^{\infty}(\mathbb{R}^2))$ emanating from $u_{0,n}$ $n\geq 1$, sharing the same existence time
\begin{equation}\label{eqexistime}
0<T\leq (1+A_s \left\|u_0\right\|_{H^s})^{-2},
\end{equation}
satisfying
\begin{equation}\label{eqexist3}
    \left\|u_n\right\|_{L^{\infty}_T H^{s}} \leq 2 \left\|u_0\right\|_{H^{s}}, 
\end{equation}
and
\begin{equation}\label{eqexist4}
   K= \sup_{n\geq 1}\left\{\left\|u_n\right\|_{L^{1}_TL^{\infty}_x} +\left\|\nabla u_n\right\|_{L^{1}_TL^{\infty}_x} \right\} <\infty.
\end{equation}
Therefore, setting $v_{n,m}=u_n-u_m$, we find
\begin{equation}\label{eqexiscauchy}
    \partial_t v_{n,m} -D_x^{\alpha+1}\partial_x v_{n,m}+\partial_x\partial_y^2 v_{n,m}+\frac{1}{2}\partial_x((u_n+u_m)v_{n,m})=0,
\end{equation}
with initial condition $v_{n,m}(0)=u_{0,n}-u_{0,m}$. Thus, by employing recurrent energy estimates, \eqref{eqexist4} and \eqref{estimapprox} we find 
\begin{equation}\label{eqdecayn}
    n^{s-\sigma}\, \left\|J^{\sigma}(u_n-u_{m})\right\|_{L^{\infty}_T L^2_{xy}} \lesssim n^{s-\sigma}e^{cK}\left\|J^{\sigma}(u_{0,n}-u_{0,m})\right\|_{L^{\infty}_T L^2_{xy}} \underset{n\to \infty}{\rightarrow}0,
\end{equation}
$m\geq n\geq 1$, whenever $0\leq\sigma<s$. 

\begin{prop}\label{exisprop1} 
Let $m\geq n \geq 1$, then
\begin{equation}\label{eqexist5}
 n\left\|u_n-u_m\right\|_{L^1_TL^{\infty}_{xy}}+\left\|\nabla(u_n-u_m)\right\|_{L^1_TL^{\infty}_{xy}} \underset{n\to \infty}{\rightarrow} 0,
\end{equation}
provided that $A_s$ in \eqref{eqexistime} is taken large enough. Moreover,
\begin{equation}\label{eqexist5.1}
    \left\|u_n-u_m\right\|_{L^{\infty}_T H^s} \underset{n\to \infty}{\rightarrow} 0.
\end{equation}
\end{prop}

\begin{proof}
We begin deducing the first estimate on the l.h.s of \eqref{eqexist5}. Let $\widetilde{s}\in(s_{\alpha},s)$ and $0<\delta < \min\{1,(s-\widetilde{s})/2\}$ fixed. An application of Lemma \ref{refinStri} with equation \eqref{eqexiscauchy}  yields 
\begin{equation}\label{eqexist6}
\begin{aligned}
\|&v_{n,n}\|_{L^1_T L^{\infty}_{xy}} \\
&\lesssim T^{1/2}\big(\|J^{\widetilde{s}-1+\delta}v_{n,m}\|_{L^{\infty}_TL^2_{xy}}+\int_0^T \|J^{\widetilde{s}-1+2\delta}((u_n+u_m)v_{n,m})(\tau)\|_{L^2_{xy}}\, d\tau \big) \\
& \lesssim T^{1/2}\big(\|J^{\widetilde{s}-1+\delta}v_{n,m}\|_{L^{\infty}_TL^2_{xy}}\\
&\hspace{2cm}+\int_0^T \big(\|J^{\widetilde{s}-1+2\delta}(u_n+u_m)(\tau)\|_{L^2_{xy}}\|v_{n,m}(\tau)\|_{L^{\infty}_{xy}} \\
&\hspace{3cm}+ \|J^{\widetilde{s}-1+2\delta}v_{n,m}(\tau)\|_{L^2_{xy}}\|(u_n+u_m)(\tau)\|_{L^{\infty}_{xy}}\big)\, d\tau 
 \big),
\end{aligned}
\end{equation}
where we have employed Lemma \ref{fraLR}. Notice that our choice of $\delta$ and \eqref{eqdecayn} give
\begin{equation*}
\|J^{\widetilde{s}-1+\delta}v_{n,m}\|_{L^{\infty}_TL^2_{xy}}=o(n^{-1}),
\end{equation*}
so that \eqref{eqexist6}, \eqref{eqexist4} and \eqref{eqexist5} imply
\begin{equation*}
\|v_{n,n}\|_{L^1_T L^{\infty}_{xy}} =  o(n^{-1})+O(T^{1/2}\|v_{n,n}\|_{L^1_T L^{\infty}_{xy}}).
\end{equation*}
Hence, taking $0<T<1$ small with respect to the above constant (that is, $A_s$ large in \eqref{eqexistime}), we find
\begin{equation}\label{eqexist7}
\|v_{n,n}\|_{L^1_T L^{\infty}_{xy}} =  o(n^{-1}).
\end{equation}
On the other hand, by a similar reasoning dealing with \eqref{eqexist6}, employing \eqref{STREeq3} and \eqref{STREeq3.2} with equation  \eqref{eqexiscauchy}, we obtain
\begin{equation*}
\begin{aligned}
\|\nabla v_{n,m}\|_{L^1_T L^{\infty}_{xy}} &\lesssim T^{1/2}\big(\|J^{\widetilde{s}+\delta}v_{n,m}\|_{L^{\infty}_TL^2_{xy}}+\|J^{\widetilde{s}+2\delta}(u_n+u_m)\|_{L^{\infty}_TL_{xy}^2}\|v_{n,m}\|_{L^1_T L^{\infty}_{xy}} \\
&\hspace{3.5cm}+\|J^{\widetilde{s}+2\delta}v_{n,m}\|_{L^{\infty}_TL_{xy}^2}\|v_{n}+v_m\|_{L^1_T L^{\infty}_{xy}}\big)\\
&=o(1)+O(\|v_{n,m}\|_{L^1_T L^{\infty}_{xy}}),
\end{aligned}
\end{equation*}
where we have employed \eqref{eqexist3} and \eqref{eqexist4}. Then, \eqref{eqexist7} completes the deduction of \eqref{eqexist5}. 

Next, we deduce \eqref{eqexist5.1}. Applying $J^s$ to \eqref{eqexiscauchy}, multiplying then by $v_{n,m}$, integrating in space shows 
\begin{equation}\label{eqexist8}
\begin{aligned}
\frac{1}{2}\frac{d}{dt}\|J^sv_{n,m}(t)\|_{L^{2}_{xy}}^2=&-\int J^s(v_{n,m}\partial_x u_n)J^sv_{n,m}\, dxdy\\
&-\int J^s(u_{m}\partial_x v_{n,m})J^sv_{n,m}\, dx dy \\
=:& A_1+A_2.
\end{aligned}
\end{equation}
Applying the Cauchy-Schwarz inequality and the commutator estimate \eqref{conmKP},
\begin{equation*}
\begin{aligned}
|A_1|&=\big|\int [J^s,v_{n,m}]\partial_x u_n J^sv_{n,m}\, dxdy+\int v_{n,m} \partial_x J^s u_n J^sv_{n,m}\, dxdy \big| \\
&\lesssim \| [J^s,v_{n,m}]\partial_x u_n\|_{L^2_{xy}}\|J^sv_{n,m}\|_{L^2_{xy}}+\|v_{n,m}\|_{L^{\infty}_{xy}}\|\partial_x J^s u_n\|_{L^2_{xy}}\|J^s v_{n,m}\|_{L^2_{xy}} \\
& \lesssim \|\nabla v_{n,m}\|_{L^{\infty}_{xy}}\|J^s u_n\|_{L^2_{xy}}\|J^sv_{n,m}\|_{L^2_{xy}}+\|\partial_x u_n\|_{L^{\infty}_{xy}}\|J^sv_{n,m}\|_{L^2_{xy}}^2\\
&\hspace{4.5cm}+\|v_{n,m}\|_{L^{\infty}_{xy}}\|\partial_x J^s u_n\|_{L^2_{xy}}\|J^s v_{n,m}\|_{L^2_{xy}}.
\end{aligned}
\end{equation*}
Now, employing energy estimates with the equation in \eqref{BOZK} and using \eqref{eqexist3} and \eqref{eqexist4}, we get
\begin{equation*}
\|\partial_x J^s u_n\|_{L^2_{xy}} \lesssim e^{cK}\|\partial_x J^s u_{0,n}\|_{L^2_{x y}} \lesssim n.
\end{equation*}
Gathering the previous estimates, we arrive at
\begin{equation*}
\begin{aligned}
|A_1|&\lesssim \|\nabla v_{n,m}\|_{L^{\infty}_{xy}}\|J^sv_{n,m}\|_{L^2_{xy}}+\|\nabla u_n\|_{L^{\infty}_{xy}}\|J^sv_{n,m}\|_{L^2_{xy}}^2\\
&\hspace{4.5cm}+(n\|v_{n,m}\|_{L^{\infty}_{xy}})\|J^s v_{n,m}\|_{L^2_{xy}}.
\end{aligned}
\end{equation*}
On the other hand, integrating by parts
\begin{equation*}
\begin{aligned}
\begin{aligned}
|A_2|&=\big|\int [J^s,u_m]\partial_x v_{n,m} J^sv_{n,m} \, dxdy-\frac{1}{2}\int \partial_x u_m (J^sv_{n,m})^2 \, dxdy\big| \\
&\lesssim \|[J^s,u_m]\partial_x v_{n,m}\|_{L^2_{xy}}\|J^sv_{n,m}\|_{L^{2}_{xy}}+\|\partial_x u_m\|_{L^{\infty}_{xy}}\|J^sv_{n,m}\|_{L^2_{xy}}^2 \\
& \lesssim \|\nabla u_m\|_{L^{\infty}_{xy}}\|J^s v_{n,m}\|_{L^{2}_{xy}}^2+\|J^s u_m\|_{L^{2}_{xy}}\|\partial_x v_{n,m}\|_{L^{\infty}_{xy}}\|J^s v_{n,m}\|_{L^{2}_{xy}}.
\end{aligned}
\end{aligned}
\end{equation*}
Inserting the estimates for $A_1$ and $A_2$ in \eqref{eqexist8}, we get
\begin{equation*}
\begin{aligned}
\frac{1}{2}\frac{d}{dt}\|J^s v_{n,m}(t)\|_{L^2_{xy}}^2 \lesssim & (n\|v_{n,m}\|_{L^{\infty}_{xy}}+\|\nabla v_{n,m}\|_{L^{\infty}_{xy}})\|J^s v_{n,m}\|_{L^2_{xy}}\\ &+(\|\nabla u_n\|_{L^{\infty}_{xy}}+\|\nabla u_m\|_{L^{\infty}_{xy}})\|J^s v_{n,m}\|_{L^2_{xy}}^2.
\end{aligned}
\end{equation*}
Hence, applying Gronwall's inequality to the above expression, and recalling \eqref{eqexist4}, we find that there exists $c>0$ such that
\begin{equation*}
\begin{aligned}
\|J^s(u_n-u_m)\|_{L^{\infty}_T L^2_{xy}} \lesssim & \big( \|J^s(u_{0,n}-u_{0,m})\|_{L^{\infty}_T L^2_{xy}}\\
&\hspace{1cm}+(n\|v_{n,m}\|_{L^1_T L^{\infty}_{xy}}+\|\nabla v_{n,m}\|_{L^1_T L^{\infty}_{xy}})\big)e^{cK}.
\end{aligned}
\end{equation*}
Now, \eqref{eqexist5.1} is a consequence of \eqref{estimapprox} and \eqref{eqexist5}.
\end{proof}

We deduce from  Proposition \ref{exisprop1}  that $u_n$ converges to a function $u$ in 
$$C([0,T];H^s(\mathbb{R}^2)\cap L^1([0,T];W^{1,\infty}(\mathbb{R}^2)).$$ Therefore, since $u_n$ solves the integral equation
$$ u_n(t)=S(t)u_{0,n}-\frac{1}{2}\int_{0}^t S(t-\tau)\partial_{x}u_n^2(\tau)\, d\tau,$$
letting $n\to \infty$ in the sense of $C([0,T];H^{s-1}(\mathbb{R}^2))$, we conclude that $u$ also solves the integral equation associated to \eqref{BOZK}. This completes the existence part of Theorem \ref{imprwellpos}. Uniqueness is derived by using a similar energy estimate to \eqref{BasicEnerE} for the difference of two solutions, and then applying Gronwall's lemma. Finally, continuous dependence is extended by approximation with the sequence of smooth solutions $\{u_n\}$ and employing this same property from Lemma \ref{comwellp}. We refer to \cite{OscarHBO,FKP} for an explicit prove of these results. 

%%%%%%%%%%%%%%%%%%%%%%%%%%%%%%%%%%%%%%%%%%%%%%%%%%%%%%%%%%%%%%%%%%%%%%%%%%%%%%%%%%%%%%%%%%%%%%%%%%%%%%%%%%%%%%%%%%%%%%%%%%%%%%%%%%%%%%%%%%%%%%%%%%%%%%%%%%%%%%%%%%%%%%%%%%%%%%%%%%%%%%%%%%%%%%%%%%%%%%%%%%%%%%%%%%%%%%%%%%%%%%%%%%%%%%%%%

\section*{Acknowledgment}

We would like to thanks Prof. Svetlana Roundenko for reading a previous version of this work. Also, we would like to thank Prof. Felipe Linares for all his suggestions.  Additionally, we thank Prof. Claudio Mu\~{n}oz  for the comments in a previous version of this work as well as his suggestions to present the graphical description of the phenomena.

%%%%%%%%%%%%%%%%%%%%%%%%%%%%%%%%%%%%%%%%%%%%%%%%%%%%%%%%%%%%%%%%%%%%%%%%%%%%%%%%%%%%%%%%%%%%%%%%%%%%%%%%%%%%%%%%%%%%%%%%%%%%%%%%%%%%%%%

\bibliographystyle{abbrv}
	%	\nocite{*}
\bibliography{bibli}

\end{document}